\documentclass[11pt,reqno]{article}
\usepackage{amsthm,amsmath, amssymb, amsopn, amsfonts,enumitem}
\usepackage{cases}
\usepackage[margin=1in]{geometry}
\usepackage[colorlinks,citecolor=blue,urlcolor=blue]{hyperref}
\usepackage{graphicx}
\usepackage{tikz}
\usetikzlibrary{decorations.fractals}
\usetikzlibrary{patterns,arrows,shapes,decorations.pathreplacing}
\usepackage{pgfplots}
\usepackage{stmaryrd} 
\usepackage{soul,xcolor}
\setstcolor{red}
\usepackage{color}

\numberwithin{equation}{section}
\theoremstyle{plain}
\newtheorem{Definition}{Definition}[section]
\newtheorem{Remark}{Remark}[section]
\newtheorem{Theorem}{Theorem}[section]
\newtheorem{Lemma}{Lemma}[section]
\newtheorem{Proposition}{Proposition}[section]
\newtheorem{Corollary}{Corollary}[section]
\newtheorem{Assumption}{Assumption}[section]

\newcommand{\be}{\begin{equation}}
\newcommand{\ee}{\end{equation}}
\newcommand{\bee}{\begin{equation*}}
\newcommand{\eee}{\end{equation*}}
\newcommand{\bi}{\begin{itemize}}
\newcommand{\ei}{\end{itemize}}
\def \E{\mathbb{E}}

\def \N{\mathbb{N}}
\def \Z{\mathbb{Z}}
\def \P{\mathbb{P}}

\def \R{\mathbb{R}}
\def \X{\mathbb{X}}

\def\Sk{\mathfrak{S}}

\def \Cc{{\mathcal C}}
\def \Bc{{\mathcal B}}

\def \Fc{{\mathcal F}}
\def \Lc{{\mathcal L}}

\def \Gc{{\mathcal G}}

\def \Tc{{\mathcal T}}

\def \eps{\varepsilon}

\makeatletter
\newcommand{\setword}[2]{%
	\phantomsection
	#1\def\@currentlabel{\unexpanded{#1}}\label{#2}%
}
\makeatother

\title{Equilibria of Time-inconsistent Stopping for One-dimensional Diffusion Processes}

\author{Erhan Bayraktar\thanks{
		Department of Mathematics, University of Michigan, Ann Arbor, email: \texttt{erhan@umich.edu}. E. Bayraktar is partially supported by the National Science Foundation under grant DMS2106556 and by the Susan M. Smith chair.}	
\and Zhenhua Wang\thanks{
		Department of Mathematics, University of Michigan, Ann Arbor, email: \texttt{zhenhuaw@umich.edu}. }
\and Zhou Zhou\thanks{School of Mathematics and Statistics, University of Sydney, Australia, email:
	\texttt{zhou.zhou@sydney.edu.au}.}
}

\begin{document}
\maketitle

\begin{abstract}
We consider three equilibrium concepts proposed in the literature for time-inconsistent stopping problems, including mild equilibria (introduced in \cite{huang2018time}), weak equilibria (introduced in \cite{MR3880244}) and strong equilibria (introduced in \cite{MR4205889}). The discount function is assumed to be log sub-additive and the underlying process is one-dimensional diffusion. We first provide necessary and sufficient conditions for the characterization of weak equilibria. The smooth-fit condition is obtained as a by-product. Next, based on the characterization of weak equilibria, we show that an optimal mild equilibrium is also weak. Then we provide conditions under which a weak equilibrium is strong. We further show that an optimal mild equilibrium is also strong under a certain condition. Finally, we provide several examples including one showing a weak equilibrium may not be strong, and another one showing a strong equilibrium may not be optimal mild.
\end{abstract}

{
\hypersetup{linkcolor=black}
\tableofcontents
}
\newpage

\section{Introduction}
On a filtered probability space $(\Omega,\mathcal{F},(\mathcal{F}_t)_{t\geq 0},\P)$ consider the optimal stopping problem,
\begin{equation}\label{e21}
\sup\limits_{\tau\in \Tc} \E[\delta(\tau)f(X_{\tau})],
\end{equation}
where $\delta(\cdot)$ is a discount function, $X=(X_t)_t$ is a time-homogeneous one-dimensional strong Markov process, and $f(\cdot)$ is a payoff function. It is well known that when $\delta(\cdot)$ is not exponential, the problem can be time-inconsistent in the sense that an optimal stopping rule obtained today may no longer be optimal from a future's perspective. 

One way to deal with this time-inconsistency is to consider the precommitted strategy, i.e., to derive a policy that is optimal with respect to the initial preference and stick to it over the whole planning horizon even if the preference changes later; see e.g., \cite{agram2020reflected, MR3614676}. Another approach to address the time-inconsistency is to look for a sub-game perfect Nash equilibrium; given the future selves follow the equilibrium strategy, the current self has no incentive to deviate from it. For equilibrium strategies we refer to the works \cite{MR4387700, MR2746576, MR2461340, MR4328502, ekeland2006being, hernandez2020me, MR4218119, MR4288523, MR4238770, MR3738841} among others for time-inconsistent control, and \cite{MR4080735, MR4276004, ebert2018never, MR4433814, huang2018time, liang2021weak,MR4332966, bodnariu2022local, MR4124420} and the references therein for time-inconsistent stopping.

How to properly define the notion of an equilibrium is quite subtle in continuous time. There are mainly two streams of research for equilibrium strategies of time-inconsistent stopping problems in continuous time.
In the first stream of research, the following notion of equilibrium is considered.

\begin{Definition}\label{def.mild}
A closed set $S\subset\X$ is said to be  a mild equilibrium, if
\begin{numcases}{}
 \label{e3} f(x)\leq J(x,S),\quad \forall x\notin S,	\\
 \label{e4}f(x)\geq J(x,S),\quad \forall x\in S,
\end{numcases}
where
\begin{equation}\label{e22}
J(x,S):=\E^x[\delta(\rho_S)f(X_{\rho_S})]\quad\text{with}\quad\rho_{S}:=\inf\{t>0: X_t\in S\}\ \text{and}\ \E^x[\cdot]=\E[\cdot|X_0=x].
\end{equation}
\end{Definition}

This kind of equilibrium is first proposed and studied in stopping problems in the context of non-exponential discounting in \cite{huang2018time}. It is called \textit{mild equilibrium} in \cite{MR4205889} to distinguish from other equilibrium concepts. Mild equilibria are further considered in \cite{MR4067078} and \cite{MR4273542} where the time inconsistency is caused by probability distortion and model uncertainty respectively.

Note that $f(x)$ is the value for immediate stopping, and $J(x,S)=\E^x[\delta(\rho_S)f(X_{ \rho_S})]$ is the value for continuing as $\rho_S$ is the first time to enter $S$ after time $0$. As a result, the economic meaning of mild equilibria appears to be clear: in \eqref{e3} when $x\notin S$, it is better to continue and get the value $J$ rather than to stop and get the value $f$. In other words, there is no incentive to deviate from the action of ``continuing". The same reasoning seems to also apply to the other case $x\in S$ in \eqref{e4}, i.e., no incentive for changing the action from ``stopping" to ``continuing". However, this is not really captured in \eqref{e4} after a second thought: In the one dimensional diffusion (and continuous-time Markov chain) setting, under some very non-restrictive condition we have $\rho_S=0$ a.s., and thus \eqref{e4} holds trivially.
	\footnote{In multi-dimensional setting, if $x\in S^o$ then $\rho_S=0$, $\P^x$-a.s.; if $x\in\partial S$ then the identity $\rho_S=0$ requires some regularity of $\partial S$, and consequently, the verification of \eqref{e4} on the boundary may not be trivial.} That is, there is no actual deviation from stopping to continuing captured in \eqref{e4}. 

Because of this issue, mild equilibria are indeed too ``mild": the whole state space is always a mild equilibrium; in most of the examples provided in \cite{huang2018time,MR4116459,MR4273542}, there is a continuum of mild equilibria. As there are often too many mild equilibria in various models, it is natural to consider the problem of equilibrium selection.

\begin{Definition}\label{def.optimal.mild}
A mild equilibrium $S$ is said to be optimal, if for any other mild equilibrium $R$, 
$$\E^x[\delta(\rho_S)f(X_{\rho_S})]\geq \E^x[\delta(\rho_R)f(X_{\rho_R})],\quad \forall\,x\in\X.$$
\end{Definition}

Note that the optimality of a mild equilibrium is defined in the sense of pointwise dominance, which is a very strong condition. The existence of optimal equilibria is first established in \cite{MR3911711} in discrete time models. The existence result is further extended to diffusion models for one-dimensional case in \cite{MR4116459} and multi-dimensional case in \cite{MR4250561}. In particular, for the one-dimensional diffusion case, \cite{MR4116459} shows that under some general assumptions an optimal mild equilibrium exists and is given by the intersection of all mild equilibria (also see Lemma \ref{lm.optimal.mild} below). \cite{MR4116459} also provides an example indicating that in general there may exist multiple optimal mild equilibria.

In the second stream of the research for equilibrium strategies for time-inconsistent stopping in continuous time, the following notion of equilibrium is introduced:\begin{Definition}\label{def.weak}
A closed set $S\subset\X$ is said to be  a weak equilibrium, if
\begin{numcases}{}
 \label{e6} f(x)\leq J(x,S),\quad \forall x\notin S,	\\
 \label{e7}\underset{\varepsilon \searrow 0}{\liminf} \dfrac{f(x)-\E^x[\delta(\rho^\varepsilon_S) f(X_{\rho^\varepsilon_S})]}{\varepsilon} \geq 0,\quad \forall x\in S,
\end{numcases}
where
$$
\rho^\varepsilon_S:=\inf\{t\geq \varepsilon: X_t\in S \}.
$$
\end{Definition}

The weak equilibrium concept for time inconsistent stopping is proposed in \cite{MR3880244}, and further studied in \cite{MR4080735,liang2021weak,MR4332966}. Obviously, as \eqref{e4} trivially holds for one-dimensional process, a weak equilibrium is also mild. Compared to mild equilibria, the condition \eqref{e4} is replaced by \eqref{e7} for weak equilibria using a first order condition. This is analog to the first order condition criterion in time-inconsistent control. As $\rho_S^\eps\geq\eps>0$, the condition \eqref{e7} does capture the deviation from stopping to continuing, and is much stronger than \eqref{e4}. However, there is still a drawback for \eqref{e7}: when the limit is equal to zero, it is possible that for all $\eps>0$ we have $f(x) < \E^x[\delta(\rho^\varepsilon_S) f(X_{\rho^\varepsilon_S})]$, and thus there is an incentive to deviate (see \cite[Remark 3.5]{MR3626618} and \cite{MR4288523,MR4205889,MR4328502} for more details). Roughly speaking, this is similar to a critical point not necessarily being a local maximum in calculus. 

Recently, \cite{MR4205889} investigated the relation between the equilibrium concepts in these two streams of research we described above, and proposed an additional notion of equilibria:

\begin{Definition}\label{def.strong}
A closed set $S\subset\X$ is said to be  a strong equilibrium, if
\begin{numcases}{}
 \notag f(x)\leq J(x,S),\quad \forall x\notin S,	\\
 \label{e8}\exists\varepsilon(x)>0,\ \text{s.t.}\ \forall \varepsilon'\leq \varepsilon(x), f(x)-\E^x[\delta(\rho^{\varepsilon'}_S) f(X_{\rho^{\varepsilon'}_S})] \geq 0,\quad \forall x\in S.
\end{numcases}
\end{Definition}

Note that in the definition of strong equilibrium, the first order condition \eqref{e7} is replaced by a local maximum condition \eqref{e8}. This remedies the issue of weak equilibria mentioned in the above, and captures the economic meaning of ``equilibrium'' more accurately.  Such kind of equilibria is also studied in \cite{MR4288523,MR4328502} for time inconsistent control. Obviously, a strong equilibrium must be weak. In \cite{MR4205889} under continuous-time Markov chain models with non-exponential discounting, a complete relation between mild, optimal mild, weak and strong equilibria is obtained:
\begin{equation}\label{eq.one.direc}
\text{optimal mild $\subsetneqq$ strong $\subsetneqq$ weak $\subsetneqq$ mild}.
\end{equation}
In this paper we aim to establish the result \eqref{eq.one.direc} for one-dimensional diffusion models under non-exponential discounting. Compared to \cite{MR4205889}, the analysis in this paper is much more delicate. The proof in \cite{MR4205889} crucially relies on the discrete state space of the Markov chain setting, and many critical ideas and steps therein cannot be applied in our diffusion framework, where novel approaches are needed for the characterizations of weak and strong equilibria. Here we list the main contributions of our paper as follows.
\bi 
\item We provide a complete characterization (necessary and sufficient conditions) of weak equilibria. As a by-product, we show that any weak equilibrium must satisfy the smooth-fit condition when the pay-off function $f$ is smooth. This gives a much sharper result in a much more general setting as compared to the smooth-fit result obtained in \cite{MR4332966}. (See Remark \ref{rm.smooth.fit} for more details.) Moreover, in our paper $f$ need not to be smooth, and our result also indicates that the smooth-fit condition is a special case of the ``local convexity" property of weak equilibria. See Remark \ref{rm.sf.general}. Undoubtedly such results related to smooth-fit condition has no correspondence in the Markov chain framework in \cite{MR4205889}.

\item We show an optimal mild equilibrium is also a weak equilibrium. This proves that the set of weak equilibria is not empty. In terms of the mathematical method, in \cite{MR4205889} the technique for the proof of such result relies on the fact that removing a point from a stopping region changes the stopping time, which is no longer applicable in the diffusion context. A different approach is developed to overcome this difficulty.

\item We provide a sufficient condition under which a weak equilibrium is also strong. The condition is easy to verify as suggested by our examples. We also show that one may remove some ``inessential" part of an optimal mild equilibrium, and the remaining part is still optimal mild (and thus weak), and in fact strong under an additional assumption. In particular, this result implies that the smallest mild equilibrium essentially has no ``inessential" parts and thus is strong. See Theorem \ref{thm.strong.optimal} and Remark \ref{rm.sk.best}.
\ei 

The rest of the paper is organized as follows. Section \ref{sec:notations} introduces the notation and main assumptions, as well as some auxiliary results that will be used frequently throughout the paper. In Section \ref{sec:weak.chara}, we provide a complete characterization of a weak equilibrium. In Section \ref{sec:ws.optimal}, we show that an optimal mild equilibrium is a weak equilibrium. Next, in Section \ref{sec:strong.optimal} we provide a sufficient condition for a weak equilibrium to be strong. We also demonstrate how to construct a strong equilibrium from an optimal mild equilibrium by removing  ``inessential" parts. In particular, we show that the smallest mild equilibrium is strong under a mild assumption. 
Finally, three examples are provided in Section \ref{sec:eg}. The first example shows that a weak equilibrium may not be strong, while the second example shows that a strong equilibrium may not be optimal mild. The final example is about finding equilibria for the stopping problem of an American put option, which is used to demonstrate the usefulness of the results in Section~\ref{sec:strong.optimal}. Figure \ref{figure} summarizes relations between the results in this paper.
\begin{figure}[h!]
	\centering
	\includegraphics[width=17cm]{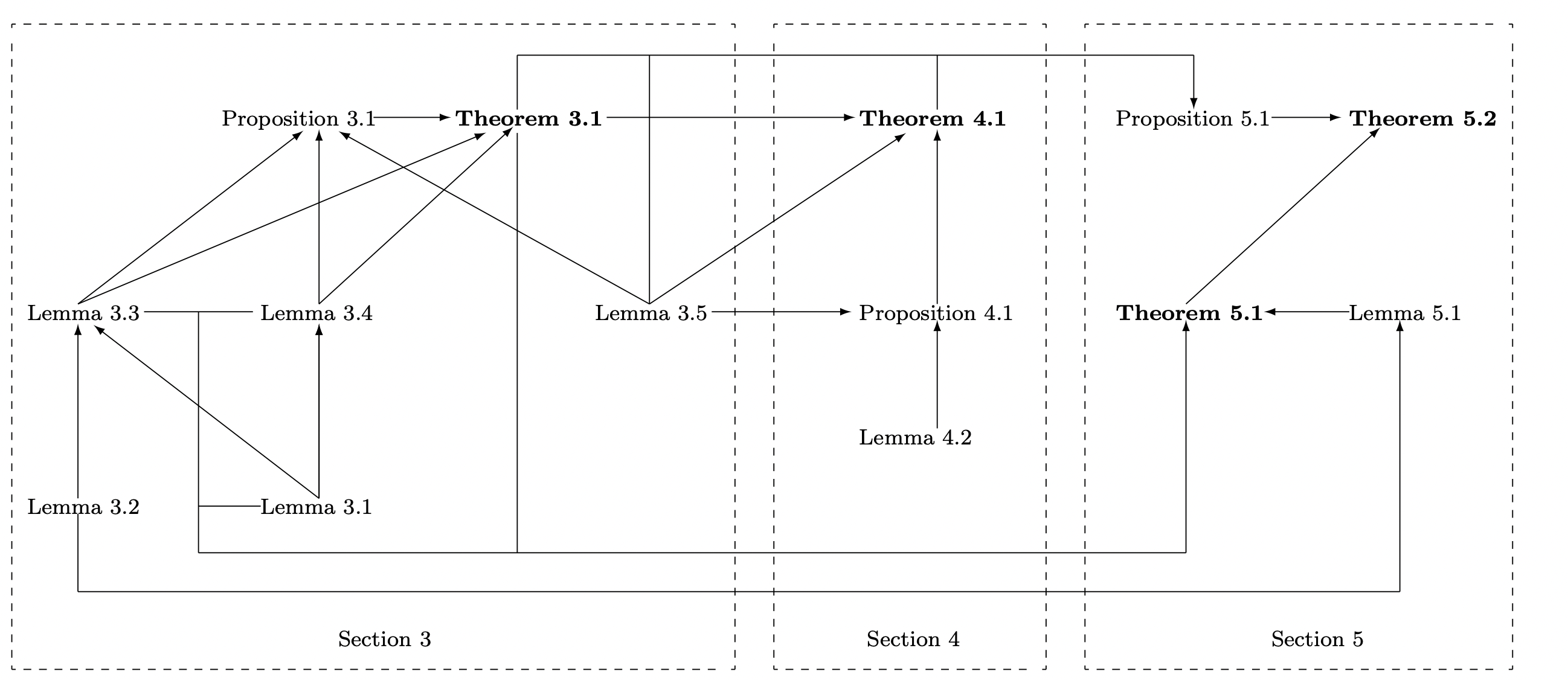}
	\caption{Relations between results in Sections \ref{sec:weak.chara}, \ref{sec:ws.optimal} and \ref{sec:strong.optimal} of this paper. $A\rightarrow B$ means that statement $A$ is used in the proof of statement $B$.}\label{figure}
	\centering
\end{figure}

\section{Setup and Some Auxiliary Results}\label{sec:notations}

Let ($\Omega$, $\Fc$, $(\Fc_t)_{t\geq 0}$, $\P$) be a filtered probability space which supports a standard Brownian motion $W=(W_t)_{t\geq 0}$. 
Let $X=(X_t)_{t\geq 0}$ be a one-dimensional diffusion process with the dynamics
\begin{equation}\label{e1}
dX_t=\mu(X_t)dt+\sigma(X_t)dW_t,
\end{equation}
and take values in an interval $\X\subset \R$. Let $\P^x$ be the probability measure given $X_0=x$ and denote $\E^x[\cdot]=\E[\cdot|X_0=x]$. Let $L^x_t$ be the local time of $X$ at point $x$ up to time $t$. Denote by $\mathcal{T}$ the set of stopping times.

Let $\Bc$ be the family of all Borel subsets within $\X$. For any $A\in \Bc$, denote $A^c:= \X\setminus A$ and $\partial A:= \overline{A}\setminus A^\circ$, where $A^\circ$ is the interior of $A$ and $\overline{A}$ is the closure of $A$ under the Euclidean topology within $\X$. Denote $B(x,r):=(x-r,x+r)\cap\X$. For $A\in\Bc$, we define the first hitting and exit times
\begin{equation}\label{e5}
	\rho_{A}:=\inf\{t>0: X_t\in A\}\quad\text{and}\quad\tau_{A}:=\inf\{t>0: X_t\notin A\}=\rho_{A^c}.
\end{equation}
Given a stopping region $A\in\Bc$, we define the value function $V(t,x,A):[0,\infty)\times \X\rightarrow \R$ as
\be\label{eq.V}
V(t,x,A):= \E^x[\delta(t+ \rho_A)f(X_{ \rho_A})].
\ee
Recall function $J$ defined in \eqref{e22}, we have $J(x,A)=V(0,x,A).$

Denote $\N:=\{1,2,\dotso\}$ and $\Z:=\{0,\pm1,\pm2,\dotso\}$. Given $E\in \Bc$ and $k\in \N\cup\{0\}$, denote by $\Cc^{1,k}([0,\infty)\times E)$ the family of functions $v(t,x)$ that are continuously differentiable with respect to (w.r.t.) $t$ and $k$-times continuously differentiable w.r.t. $x$ when restricted to $[0,\infty)\times E$, and $\Cc^k(E)$ the family of functions $v(x)$ that are $k$-times continuously differentiable when restricted to $E$. \footnote{Continuous differentiability is extended to the boundary in a natural way if the boundary is included in $E$. For example, given $[a,b]\subset \X$, we say $g\in \Cc^{1,2}([0,\infty)\times [a,b])$, if $g=g_1$ on $[0,\infty)\times [a,b]$ for some $g_1\in \Cc^{1,2}([0,\infty)\times \X)$.} For a function $v(t,x):[0,\infty)\times E\to \R$, $v_x, v_{xx}$ (resp. $v_t$) denote the first and second order derivatives w.r.t $x$ (resp. the first order derivative w.r.t. $t$) if the derivatives exist. Moreover, denote by $v_x(t,x-)$ (resp. $v_x(t, x+)$) the left (resp. right) derivative of $v$ w.r.t. $x$ at point $(t,x)$. Similar notation applies to $v_{xx}(t,x-), v_{xx}(t,x+)$. For convenience, we denote $v_t(0,x)$ as the right derivative w.r.t. $t$ at time $t=0$. We further define the parabolic operator
$$
\Lc v(t,x):= v_t(t,x)+\mu(x) v_x(t,x)+\frac{1}{2}\sigma^2(x) v_{xx}(t,x)\quad \text{for a function } v\in \Cc^{1,2}([0,\infty)\times E).
$$
Let us also use the following notation involving left or right derivatives w.r.t. $x$:
$$ 
\Lc v(t,x\pm ):= v_t(t,x)+\mu(x) v_x(t,x\pm )+\frac{1}{2}\sigma^2(x) v_{xx}(t,x\pm ),\quad \forall t\geq 0.\\
$$

We now introduce the main assumptions in this paper. The first assumption concerns $\mu$ and $\sigma$.
\begin{Assumption}\label{assume.x.elliptic}
(i)  $\mu,\sigma: \X\to\R$ are Lipschitz continuous.
(ii) $\sigma^2(x)>0$ for all $x\in\X$.
\end{Assumption}

\begin{Remark}\label{rm.close.regular}
Assumption \ref{assume.x.elliptic}(i) guarantees that \eqref{e1} has a unique strong solution given $X_0=x\in\X$.  Assumption \ref{assume.x.elliptic}(i)(ii) together imply that for any $x\in \X$ and $t>0$,
\begin{equation}\label{e2}
\P^x\left(\min_{0\leq s\leq t}X_s<x\right)=\P^x\left(\max_{0\leq s\leq t}X_s>x\right)=1,\;	\text{and thus }\rho_{\{x\}}=0,\ \P^x\text{-a.s.}.
	\end{equation}	
A quick proof for \eqref{e2} is relegated in Appendix \ref{sec:appendix}.
\end{Remark}
 
Notice that a (time-homogeneous Markovian) stopping policy can be characterized by a stopping region $S\subset\X$. For $S\in \Bc$, \eqref{e2} implies that $\rho_S= \rho_{\overline{S}}$ $\P^x$-a.s. for any $x\in \X$. Also, $f(x)=J(x,S)$ for all $x\in \overline{S}$, and a boundary point $x$ of $S$ corresponds to the action ``immediate stopping", not matter $x$ belongs to $S$ or not.
Therefore, it suffices to work on stopping regions that are closed. 

\begin{Definition}\label{eq.def.admissible}
	$S\in\Bc$ is called an admissible stopping policy, if $S$ is closed (w.r.t the Euclidean topology within $\X$) and for any $x\in\partial S$, one the following two cases holds:
	\bi  
	\item[(\setword{a}{boundary.a})] $x\in \partial (S^\circ)$, i.e., $\exists h>0$ such that either $(x-h,x)\subset S^\circ$ and $(x,x+h)\subset S^c$, or $(x-h,x)\subset S^c$ and $(x,x+h)\subset S^\circ$;
	\item[(\setword{b}{boundary.b})] $x$ is an isolated point, i.e., $B(x,h)\setminus\{x\}\subset S^c $ for some $h>0$.
	\ei
\end{Definition}

\begin{Remark}
	Except cases (\ref{boundary.a}) \& (\ref{boundary.b}), the rest situation for a boundary point $x\in\partial S$ is the following:
	\bi
	\item[(\setword{c}{boundary.c})] There exist two sequences  $(x_n)_{n\in \N}\subset S$ and $(y_n)_{n\in \N}\subset S^c$ such that both $(x_n)_{n\in \N}$ and $(y_n)_{n\in \N}$ approach to $x$ from the left, or  both approach to $x$ from the right.\footnote{Indeed, let $x\in\partial S$. Suppose $x$ does not satisfy case (\ref{boundary.c}). Then there exists $h>0$ such that either $(x-h,x)\subset S$ or $(x-h,x)\subset S^c$, so is the interval $(x,x+h)$. If $(x-h,x)\subset S^c$ and $(x,x+h)\subset S^c$, then $x$ satisfies case (\ref{boundary.b}); otherwise $x$ satisfies case (\ref{boundary.a}).}
	\ei
	Stopping regions containing boundary case (\ref{boundary.c}) lack economic meaning, since it is not practical for an agent to follow a stopping policy classified as case (\ref{boundary.c}). Mathematically, the regularity of $V(t,x,S)$ may also be missing when $S$ contains boundary case (\ref{boundary.c}), e.g., $V_x(t,0+,S)$ may not exist for $S$ being the cantor set on $[0,1]$; this would cause serious issue to establish our main results later as they crucially rely on the regularity of $V(t,x,S)$ (e.g., the characterization of weak equilibria).
	
	Focusing on admissible stopping policies is also well aligned with the literature, and a stopping policy containing boundary case (\ref{boundary.c}) is rarely studied in applications. For instance, all the case studies in \cite{bodnariu2022local, MR4332966, MR3880244, MR4124420, MR4067078} only focus on threshold-type equilibria. The results in \cite{ebert2018never} mainly focuses on two threshold stopping regions. The mild equilibria in all the examples of \cite{huang2018time} have boundaries of cases only (\ref{boundary.a}) and (\ref{boundary.b}). All mild equilibria provided in \cite[Sections 6.1 and 6.2]{MR4116459} are all admissible, so are the mild equilibria in the case study of \cite[Section 4]{MR4273542}.
	
	Let us also point out that all the interesting equilibria (i.e., optimal mild, weak, strong equilibria) provided in all the examples in this paper are admissible. Specifically, in the case study in Section \ref{subsec:eg.GBM}, which can be thought of as a continuation of \cite[Sections 6.3]{MR4116459}, all the weak, strong and optimal mild equilibria are admissible, and any mild equilibria is either admissible or has an admissible alternative (see Remark \ref{rm:eg.admissible}). 
	
	To sum up, focusing on cases (\ref{boundary.a}) and (\ref{boundary.b}) is economically meaningful, mathematically necessary, well-aligned with the literature, and general enough for applications.
\end{Remark}

Let $\delta(\cdot): [0,\infty)\rightarrow [0,1]$ be a discount function that is non-increasing, continuously differentiable, and $\delta(0)=1$, $\delta(t)<1$ for $t>0$. We assume $\delta$ satisfies the following condition. 

\begin{Assumption}\label{assume.delta.deriv}
$\delta$ is log sub-additive:
\be\label{eq.assume.DI} 
\delta(t+s)\geq \delta(t)\delta(s),\quad \forall s,t\geq 0.
\ee
\end{Assumption}

\begin{Remark} Condition \eqref{eq.assume.DI} can be interpreted as the so-called \textit{decreasing impatience} in finance and economics. Many non-exponential discount functions, including hyperbolic, generalized hyperbolic and pseudo-exponential discounting, satisfy \eqref{eq.assume.DI}. See the discussion below \cite[Assumption 3.12]{huang2018time} for a more detailed explanation.
\end{Remark}
Recall that $\delta'(0)$ denotes the right derivative of $\delta(t)$ at $t=0$. The following lemma is a quick result for $\delta$ and the proof is relegated in the Appendix \ref{sec:appendix}.
\begin{Lemma}\label{lm.delta.inequ}
	Let Assumption \ref{assume.delta.deriv} hold. Then 
$$
\delta'(t)\geq \delta(t)\delta'(0),\; \; \text{and}\;\; 1-\delta(t)\leq |\delta'(0)|t,\quad \forall t\geq 0.
$$
\end{Lemma}

Let the payoff function $f(x):\X\rightarrow \R$ be non-negative and continuous. We further assume $f$ satisfies the following assumptions.

\begin{Assumption}\label{assume.f.positive}
(i) For any $x\in\X$,
\be\label{eq.assume.sstar'} 
	\lim_{t\to\infty}\delta(t)f(X_t)=0,\quad  \P^x-a.s.,
	\ee  
and there exists $\zeta>0$ such that
\be\label{eq.assume.wellpose}  
\E^x\left[\sup_{t\geq 0} \left(\delta(t)f(X_t)\right)^{1+\zeta}\right]<\infty.
\ee	

\noindent (ii) $f(x)$ belongs to $\Cc^2$ piecewisely. That is, there exists an either finite or countable set $(\theta_{n})_{n\in I}\subset\X$, with $I\subset \Z$ and $\theta_n<\theta_{n+1}$ for all $n\in I$, such that $f\in \Cc^2([\theta_n, \theta_{n+1}])$ for any $n\in I$. We also assume that $\inf_{n\in I}(\theta_{n+1}-\theta_n)>0$ and denote 
\be\label{eq.assume.G} 
\mathcal{G}:= \X\setminus\{\theta_n: n\in I\}.
\ee 
\end{Assumption}

\begin{Remark}\label{rm.f.wellpose}
The assumption \eqref{eq.assume.wellpose} will be used for Lemma \ref{lm.varepsilon.h}, which is an essential lemma for all the main results in the paper. Moreover, \eqref{eq.assume.wellpose} implies that
\be\label{eq.assume.wellpose1}    
\E^x\left[\sup_{t\geq 0} \delta(t)f(X_t)\right]<\infty,\quad \forall x\in\X.
\ee 
This together with \eqref{eq.assume.sstar'} guarantees the well-posedness of $V(t,x,S)$ for any stopping policy $S$.\footnote{\eqref{eq.assume.sstar'} is used for the well-posedness of $V(t,x,S)$, because otherwise $\delta(t)f(X_t)$ is not well defined on $\rho_S=\infty$ (unless we do some extension, e.g. by considering the upper limit $\limsup_{t\to\infty}\delta(t)f(X_t)$.}  \eqref{eq.assume.wellpose1} and \eqref{eq.assume.sstar'} will also be used for applying the dominated convergence theorem in some localization arguments in the proofs later. Furthermore, \eqref{eq.assume.wellpose1} and \eqref{eq.assume.sstar'} also ensure the existence of an optimal mild equilibrium as demonstrated in \cite[Theorem 4.12]{MR4116459} (also see Lemma \ref{lm.optimal.mild} in this paper).
\end{Remark}

Let us make an assumption on $V(t,x,S)$.
\begin{Assumption}\label{assume.new}
	For any admissible stopping policy $S$ and $a,b\in\X$ with $a<b$ and $(a,b)\subset S^c$, $V(t,x,S)$ defined in \eqref{eq.V} (with $A=S$) belongs to $\Cc^{1,2}([0,\infty)\times[a,b])$, and
	\be\label{eq.vx.lineart}
	\limsup_{t\searrow 0}\frac{1}{\sqrt{t}}|V_x(t,x\pm, S)-V_x(0,x\pm, S)|=0,\quad\forall\,x\in\X.
	\ee
\end{Assumption}

\begin{Remark}
It turns out Assumption \ref{assume.new} is quite general.  A sufficient condition for Assumption \ref{assume.new} is that $\delta(t)$ is a \textit{weighted discount function} as shown in the lemma below. One may also directly verify this assumption given the probability density functions of exit time
\begin{equation}\label{eq790}
p(x,t):=\P^x\left(\tau_{(c,d)}\in dt, X_{\tau_{(c,d)}}=c\right)\quad\text{and}\quad q(x,t):=\P^x\left(\tau_{(c,d)}\in dt, X_{\tau_{(c,d)}}=d\right)
\end{equation}
being regular enough. For example, if $X$ is a Brownian motion on $\X=\R$, $\delta(t)=\frac{1}{1+t}$, and $f(x)=0\vee x$, then we can verify Assumption \ref{assume.new} holds by using \eqref{eq790} for the Brownian motion. Providing a more general sufficient condition for Assumption \ref{assume.new} is out of the scope of this paper.
\end{Remark} 

\begin{Lemma}\label{rm.suff.newassume}
Let Assumption \ref{assume.x.elliptic} hold and $f$ be bounded on $\X$. Suppose $\delta(t)$ is a \textit{weighted discount function} of the following form
	\be\label{eq.weight.delta} 
	\delta(t)=\int_0^\infty e^{-rt}dF(r), 
	\ee
	where $F(r):[0,\infty)\rightarrow [0,1]$ is a cumulative distribution function satisfying $\int_0^\infty rdF(r)<\infty$ and
	\be\label{eq.weight.at} 
	\lim_{t\searrow 0}\frac{1 }{\sqrt{t}} \int_0^\infty r(1-e^{-rt})dF(r)=0.
	\ee
Then Assumption~\ref{assume.new} holds.
\end{Lemma}
The proof of Lemma \ref{rm.suff.newassume} is included in Appendix \ref{sec:appendix}.

\begin{Remark}\label{rm.weight.discount}
In \cite{MR4124420} weighted discount functions are studied in detail. \cite{MR4332966} investigates weak equilibria and the smooth-fit condition for time-inconsistent stopping in a weighted discounting setting. Many discount functions, including exponential, hyperbolic, generalized hyperbolic and pseudo-exponential discounting, satisfy \eqref{eq.weight.delta} and \eqref{eq.weight.at}. For example, a generalized hyperbolic discount function can be written as
	$$\delta(t)=\frac{1}{(1+\beta t)^{\frac{\gamma}{\beta}} }=\int_0^\infty e^{-rt} \frac{r^{\frac{\gamma}{\beta}-1} e^{-\frac{r}{\beta}}}{\beta^{\frac{\gamma}{\beta}} \Gamma(\frac{\gamma}{\beta})}dr=\int_0^\infty e^{-rt}dF(r),\quad\text{with}\ \frac{dF(r)}{dr}=\frac{r^{\frac{\gamma}{\beta}-1} e^{-\frac{r}{\beta}}}{\beta^{\frac{\gamma}{\beta}} \Gamma(\frac{\gamma}{\beta})},$$
where $\beta, \gamma>0$ are constants and $\Gamma(\cdot)$ is the gamma function (see \cite[Section 2.1]{MR4332966}). A direct calculation shows that
	\begin{align*}
		\int_0^\infty r(1-e^{-rt})dF(r)=\int_0^\infty r(1-e^{-rt})\frac{r^{\frac{\gamma}{\beta}-1} e^{-\frac{r}{\beta}}}{\beta^{\frac{\gamma}{\beta}} \Gamma(\frac{\gamma}{\beta})}dr
		=\gamma-\gamma\frac{1}{(1+\beta t)^{\frac{\gamma}{\beta}+1}}
		\leq  \gamma(\gamma+\beta)t\quad \forall t> 0,
	\end{align*}
which implies \eqref{eq.weight.at}.
	\end{Remark}

The next lemma summarizes several preliminary properties of $V(t,x,S)$ which will be used to establish the main results in later sections.

\begin{Lemma}\label{lm.v.c12}
	 Let Assumptions \ref{assume.x.elliptic}, \ref{assume.delta.deriv}, \ref{assume.f.positive}(ii), \ref{assume.new} hold and $S$ be an admissible stopping policy. Then
	\bi  
	\item[(a)] $V(t,x,S)$ belongs to $\Cc^{1,2}([0,\infty)\times \overline{S^c})$, and $V(t,x,S)=\delta(t) f(x)$ for any $(t,x)\in [0,\infty)\times S$. Moreover,
	\be\label{eq.lm.lv0} 
	\Lc V(t,x,S)\equiv 0, \quad \forall (t,x)\in [0,\infty)\times S^c.
	\ee 
	\item[(b)]  $\Lc V(t,x\pm,S)$ exists for all $(t,x)\in [0,\infty)\times \X$. For any $h>0$ and $x_0\in \X$ such that $\overline{B(x_0,h)}\subset \X$,  we have that
	\bee  
	\sup\limits_{(t,x)\in [0, \infty)\times \overline{B(x_0, h)}}|\Lc V(t,x\pm,S)|<\infty.
	\eee
	\ei 
\end{Lemma}
The proof of Lemma \ref{lm.v.c12} is provided in Appendix \ref{sec:appendix}. Throughout this paper, we will keep using the following local time integral formula provided in \cite{MR2408999}.
\begin{Lemma}\label{lm.peskir.ito}
Let $a, x_0, b\in\R$ with $a<x_0<b$. Suppose $g(t,y): [0,\infty)\times \R\rightarrow \R$ such that $g\in \Cc^{1,2}((0,\infty)\times (a, x_0])$, $g\in \Cc^{1,2}((0,\infty)\times[x_0, b))$. Then for $X_0=x\in (a,b)$, we have that
	\begin{align*}
		g(t,X_t)=&g(0,x)+\int_0^t \frac{1}{2}(\Lc g(s,X_{s}-)+\Lc g(s, X_s+))ds
		+\int_0^t g_x(s, X_s)\sigma(X_s)\cdot 1_{\{X_s\neq x_0 \}}dW_s\\
		&+\frac{1}{2}\int_0^t(g_x(s,x_0+)-g_x(s,x_0-))dL^{x_0}_s\quad \forall 0\leq t\leq \tau_{(a,b)}.
	\end{align*}
\end{Lemma}

\section{Characterization for Weak Equilibria}\label{sec:weak.chara}
In this section, we provide the characterization for weak equilibria. Such characterization is critical to study of the relations between mild, weak and strong equilibria. Below is the main result of this section.

\begin{Theorem}\label{prop.weak.cha}
	Let Assumptions \ref{assume.x.elliptic}--\ref{assume.new} hold. Suppose $S$ is an admissible stopping policy. Then $S$ is a weak equilibrium if and only if the followings are satisfied.
\begin{align}
	\label{eq789} &	V(0,x,S)\geq f(x)\quad \forall x\notin S;\\
	\label{eq.weak.equicha1} & 	V_x(0,x-,S)\geq V_x(0,x+,S)\quad \forall  x\in S;\\
	\label{eq.weak.equicha2}  &\Lc V(0,x-,S)\vee \Lc V(0,x+,S)\leq 0\quad  \forall x\in \X.
\end{align}
\end{Theorem}
The proof of Theorem \ref{prop.weak.cha} will be presented in the next subsection.
A consequence of Theorem \ref{prop.inside.cha} is the following smooth-fit condition of $V$ at the boundary $\partial S$ when $f$ is smooth.
\begin{Corollary}[Smooth-fit condition for weak equilibria when $f$ is smooth]\label{cor.boundary.smoothfit}
Let Assumptions \ref{assume.x.elliptic}--\ref{assume.new} hold, and let $S$ be an admissible stopping policy. Suppose $S$ is a weak equilibrium. Then for any $x\in\partial S$, if $f'(x)$ exists, then
$V_x(0,x-,S)=V_x(0,x+,S).$
\end{Corollary}

\begin{proof}
Take an arbitrary $x\in \partial S$. Take $x\in \partial S$. By Theorem \ref{prop.inside.cha}, it suffices to prove that $V_x(0,x-,S)\leq V_x(0,x+,S)$ for both boundary cases (\ref{boundary.a}) and (\ref{boundary.b}).

Recall $\Gc$ defined in \eqref{eq.assume.G}. For boundary case (\ref{boundary.a}), without loss of generality, we assume $(x,x+h)\subset (S^\circ\cap \Gc)$ and $(x-h,x)\subset S^c$ for some $h>0$. Since $V(0,x,S)\geq f(x)$ on $S^c$ by \eqref{e6} and $V(0,x,S)=f(x)$ on $S$ by Lemma \ref{lm.v.c12} (a), we have that for $\varepsilon>0$ small enough,
$$
\frac{V(0,x-\varepsilon,S)-V(0,x,S)}{\varepsilon}\geq \frac{f(x-\varepsilon)-f(x)}{\varepsilon}.
$$
By the differentiability of $V$ on $[0,\infty)\times \overline{S^c}$ (due to Lemma \ref{lm.v.c12}(a)) and existence of $f'(x)$, the above inequalities implies that
$$
V_x(0,x-,S)\leq f'(x-)=f'(x+)=V_x(0,x+,S),
$$
where the last equality follows from $V(0,x,S)=f(x)$ on $(x,x+h)\subset (S^\circ\cap \Gc)$.

For boundary case (\ref{boundary.b}), we can choose a constant $h>0$ such that $(B(x,h)\setminus\{x\})\subset (S^c\cap \Gc)$. Then $V(0,y,S)\geq f(y)$ for all $y\in B(x,h)\setminus\{x\}$, which implies that
$$
V_x(0,x-,S)\leq f'(x-)=f'(x+)\leq V_x(0,x+,S)
$$
by an argument similar to that for boundary case (\ref{boundary.a}).
\end{proof}

\begin{Remark}\label{rm.smooth.fit}
In \cite{MR4332966}, it is shown that with the underlying process being a geometric Brownian motion, the smooth-fit condition together with some inequalities provides a weak equilibrium; in addition, the real options example in \cite{MR4332966} indicates that when smooth-fit condition fails, there is no weak equilibrium. This, however, does not indicate whether any weak equilibrium must satisfy the smooth-fit condition. Here we are able to provide a much sharper result in a much more general setting: given $f$ is smooth, any weak equilibrium must satisfy the smooth-fit condition, and may be constructed by the smooth-fit condition together with some other related inequalities. Let us also mention that smooth-fit result is also established in a very recent paper \cite{bodnariu2022local} for {\it mixed } weak equilibrium under a general setting.
\end{Remark}

\begin{Remark}\label{rm.sf.general}
In our paper, the payoff function $f$ is only required to be piecewisely smooth. The inequality in \eqref{eq.weak.equicha1} and Corollary \ref{cor.boundary.smoothfit} show that the smooth-fit condition is a specially case of the ``local convexity" property for a weak equilibrium $S$: the left derivative w.r.t. $x$ of the value function $V(0,x,S)$ must be bigger than or equal to its right derivative for \textit{any} $x\in S$. In particular, if the payoff function is smooth at a point $x\in S$, such convexity property is reduced to the smooth-fit condition.\end{Remark}

\begin{Remark}\label{rm.exponential}
Suppose the discount function is exponential in the current one-dimensional diffusion context. Then \eqref{eq789} and \eqref{eq.weak.equicha2} together yield the variational inequalities. As is well known in classical optimal stopping theory, (under suitable assumptions) the optimal stopping value and strategy can be characterized by variational inequalities. Therefore, when the discount function is exponential, Theorem \ref{prop.weak.cha} indicates that any weak equilibrium is an optimal stopping region in the classical sense, so are strong and optimal mild equilibrium (as we will show later that an optimal mild equilibrium is also weak). On the other hand, a mild equilibrium is not necessarily a classical optimal stopping region, e.g., the whole state space $\X$ is a mild equilibrium but may not be an optimal stopping region in general.
\end{Remark}

\subsection{Proof of Theorem \ref{prop.weak.cha}}\label{subsec:3.1}

To characterize a weak equilibrium, one shall consider the two conditions \eqref{e6} and \eqref{e7} in Definition \ref{def.weak}. \eqref{e6} is the same as \eqref{eq789} and thus we will focus on condition \eqref{e7}. By $V$ defined in \eqref{eq.V}, \eqref{e7}  can be rewritten as 
\be\label{eq.explain.31.0}  
\limsup\limits_{\varepsilon\searrow 0}\dfrac{1}{\varepsilon} \Big(\E^x[\delta( \rho^\varepsilon_S)f(X_{\rho^\varepsilon_S})]-f(x)\Big)=\limsup\limits_{\varepsilon\searrow 0}\dfrac{1}{\varepsilon} \Big(\E^x[V(\varepsilon,X_\varepsilon,S)]-V(0,x,S)\Big)\leq 0, \quad x\in S.
\ee 
Since $X_\eps$ and thus $V(\varepsilon,X_\varepsilon,S)$ are not uniformly bounded, we will apply some localization argument and restrict $X$ within a bounded ball $B(x,h)$. Moreover, as $x\mapsto V(t,\cdot, S)$ is only piecewisely smooth, we will choose $h>0$ small enough, such that $V_x$ is only (possibly) discontinuous at the center of the ball $B(x,h)$, in order to apply Lemma \ref{lm.peskir.ito} to $V$ in \eqref{eq.explain.31.0}. By doing so, we will end up with
\begin{align}
\label{eq.explain.31.1} \E^x[V(\eps, X_\eps, S)-V(0,x,S)]\approx &\E^x[V(\eps\wedge \tau_{B(x,h)}, X_{\eps\wedge \tau_{B(x,h)}})-V(0,x,S)]\\
\notag=&\E^x\left[\int_0^{\varepsilon\wedge \tau_{B(x,h)}} \frac{1}{2}(\Lc V(s,X_{s}-,S)+\Lc V(s, X_s+,S))ds\right]\\
\label{eq.explain.31.2} 	&+\E^x\left[\frac{1}{2}\int_0^{\varepsilon\wedge \tau_{B(x,h)}}(V_x(s,x+,S)-V_x(s,x-,S))dL^x_s\right].
\end{align}
where the approximation in \eqref{eq.explain.31.1} will be made rigorous in Lemma \ref{lm.varepsilon.h}, which is built on Lemma \ref{lm.step.1}, and \eqref{eq.explain.31.2} is due to Lemma \ref{lm.peskir.ito}.
Then the condition \eqref{eq.explain.31.0} boils down to comparing the two integral terms on the right-hand-side (RHS) of \eqref{eq.explain.31.2}. This requires estimates for the expected local time $\E^x[L^x_{\eps\wedge \tau_{B(x,h)}}]$ (see Lemma \ref{lm.local.time}, which is built upon Lemmas \ref{lm.step.1} and \ref{lm.step.2}) and the growth of $V_x(t,x\pm,S)$ w.r.t. $t$ (see Lemma \ref{lm.v.lrderiv}). This is the overall idea on how we obtain Theorem \ref{prop.weak.cha}.
 
Throughout this section, we shall also take advantage of the following standard estimate for moments of diffusions (see e.g., \cite[Problem 3.15 on page 306]{MR1121940}): Given a process $Z_t$ satisfying $dZ_t= \beta(Z_t)dt+\delta(Z_t)dW_t$ with $\beta, \delta$ being Lipschitz and $Z_0=z\in \X$, for all $0\leq \eps\leq 1$ and $m\geq 1$ it holds that, 
\begin{align}
&\E^x\left[ \sup_{0\leq s\leq t} |Z_s|^{2m} \right]<\infty\quad \forall t\in(0,\infty)  \label{e10'},\\
&\E^x\left[|Z_\varepsilon-z|^{2m}\right]\leq K(1+|z|^{2m}) \varepsilon^m \label{e10},
\end{align}
where $K$ is a constant independent of $\eps$.

We first provide two Lemmas dealing with the probability of $X$ exiting a ball, and the first order moment related to $X$ over a small time horizon $\eps$. They will be used for proofs in both the current and later sections.

\begin{Lemma}\label{lm.step.1}
Let Assumption \ref{assume.x.elliptic} hold. For any fixed $a>0$ we have that
\be\label{eq.lm.step1}  
	\P^{x}(\tau_{B(x,h)}\leq \eps)=o(\eps^a),\quad \text{for $\eps> 0$ small enough.}
	\ee
\end{Lemma}

\begin{proof}
	Fix $a>0$. We invoke the ``change of space" method in \cite[Section 5.2]{MR2256030}. Consider the process
	\be\label{eq.lm.change}  
	Y_t:= \phi (X_t), Y_0:=\phi(x) \; \text{with} \;\phi(y):= \int_0^y\exp\left( -\int_0^l \frac{2\mu(z)}{\sigma^2(z)}dz \right)dl.
	\ee 
	Thanks to Assumption \ref{assume.x.elliptic}, $\phi$ is well-defined, strictly increasing, and has first and second derivatives. A direct calculation shows that $dY_t=\sigma(X_t)\phi'(X_t)dW_t$, and the exit time to $\overline{B(x,h)}$ of $X_t$ is equivalent to the exit time of $Y_t$ to the interval $[\phi(x-h), \phi(x+h)]$. Set $\tilde{h}:=\left(\phi(x+h)-\phi(x) \right)\wedge \left(\phi(x)-\phi(x-h) \right)>0$ and $\tilde{a}:= a+1$. Let $0<\eps\leq 1$. We have that
	\be\label{eq.lm.1} 
	\P^{x}(\tau_{B(x,h)}\leq \eps)\leq \P^{Y_0} \left(\sup_{0\leq t\leq \eps} |Y_t-Y_0| \geq \tilde{h} \right)=\P^{Y_0} \left( \sup_{0\leq t\leq \eps} |Y_t-Y_0|^{2\tilde{a}} \geq \tilde{h}^{2\tilde{a}}\right).
	\ee 
Notice that $Y$ is a martigale (within the interval $\overline{B(\phi(x),\tilde{h})}$), we can then apply the Doob's submartingale inequality to the RHS of \eqref{eq.lm.1} to conclude that
	\be\label{eq.lm.probepsilonh}
	\begin{aligned}
		&\P^{Y_0} \left( \sup_{0\leq t\leq \eps} |Y_t-Y_0|^{2\tilde{a}} \geq \tilde{h}^{2\tilde{a}}\right)\leq  \frac{\E^{Y_0}[|Y_\eps-Y_0|^{2\tilde{a}}]}{\tilde{h}^{2\tilde{a}}}\leq \frac{\tilde{K}(1+\phi^{2\tilde{a}}(x))\eps^{\tilde{a}}}{{\tilde{h}}^{2\tilde{a}}}, 
	\end{aligned}
	\ee
	where the last inequality follows from \eqref{e10}, and $\tilde{K}$ is a positive constant independent of $\eps$. Then \eqref{eq.lm.step1} follows from \eqref{eq.lm.1}, \eqref{eq.lm.probepsilonh} and the fact that $\tilde{a}>a$.
\end{proof}

\begin{Lemma}\label{lm.step.2}
Let Assumption \ref{assume.x.elliptic}(i) hold. For $\eps>0$ small enough we have that
\begin{equation}\label{e11}
	\E^{x}[|\bar{X}_\varepsilon|]=O(\varepsilon),\quad \text{with }	\widetilde{X}_t:=x+\mu(x)t+\sigma(x) W_t \text{ and } \bar{X}_t:= X_t-\widetilde{X}_t.
\end{equation}
\end{Lemma}

\begin{proof}
Throughout the proof, $C$ will serve as a generic constant may change from line to line but is independent of $\eps$. Let $0<\eps\leq 1$. First, we have
	\be\label{eq.xbar} 
	\begin{aligned}
		\E^{x}|\bar{X}_\eps|\leq \E^{x}\left|\int_0^\varepsilon(\mu(X_t)-\mu(x))dt\right|+\E^{x}\left|\int_0^\varepsilon (\sigma(X_t)-\sigma(x))dW_t\right|.
	\end{aligned}
	\ee 
	By applying \eqref{e10} on $X_t$ with $m=1$, we have
	$
	\E^{x}[|X_\varepsilon-x|^{2}]\leq C \varepsilon.
	$
	This together with the Lipschitz continuity of $\mu$ implies
	\be\label{eq.lm.step3} 
	\begin{aligned}
		\E^{x}\left|\int_0^\varepsilon(\mu(X_t)-\mu(x))dt\right|&\leq \E^{x}\left[\int_0^\varepsilon \frac{1}{2}(1+|\mu(X_t)-\mu(x)|^2)dt\right]\\
		&\leq \frac{1}{2}\varepsilon +\frac{1}{2}  \int_0^\varepsilon C\E^{x}[|X_t-x|^2]dt=O(\eps).
	\end{aligned} 
	\ee 
	Similarly, we can estimate the second term in \eqref{eq.xbar} as follows
	\be\label{eq.lm.step3'} 
	\begin{aligned}
		\E^{x}\left|\int_0^\varepsilon (\sigma(X_t)-\sigma(x))dW_t\right|&\leq \left(\E^{x} \left[\int_0^\varepsilon (\sigma(X_t)-\sigma(x))^2dt\right]\right)^{1/2}\\
		&\leq \left(\int_0^\varepsilon C\E^{x}[|X_t-x|^2]dt\right)^{1/2}=O(\eps)
	\end{aligned}
	\ee 
	Then by plugging \eqref{eq.lm.step3} and \eqref{eq.lm.step3'} into \eqref{eq.xbar}, we have $\E^{x}[|\bar{X}_\varepsilon|]=O(\varepsilon)$.
\end{proof}

The next lemma concerns the approximation in \eqref{eq.explain.31.1}.
\begin{Lemma}\label{lm.varepsilon.h}
	Let Assumptions \ref{assume.x.elliptic} and \ref{assume.f.positive}(i) hold. Let $S\in\Bc, x\in\X$ and $h>0$. Then for $\eps>0$ small enough,
	\be\label{eq.lm.vepsilonh} 
	\E^x[V(\varepsilon, X_\varepsilon,S)]=\E^x[V(\varepsilon\wedge \tau_{B(x,h)},X_{ \varepsilon\wedge \tau_{B(x,h)}},S)]+o(\varepsilon).
	\ee
\end{Lemma}

\begin{proof}
	Let $h>0$ and $x\in \X$. Recall the constant $\zeta$ in \eqref{eq.assume.wellpose}.
	We have that
	\be\label{eq.lm.V} 
	\begin{aligned}
		0\leq &\E^x[V(\eps, X_\eps, S)\cdot 1_{\{  \varepsilon>\tau_{B(x,h)} \}}]\\
		\leq & \left( \E^x\left[V^{1+\zeta}(\eps, X_\eps, S)\right]\right)^{\frac{1}{1+\zeta}}\cdot \left(\E^x\left[1^{\frac{1+\zeta}{\zeta}}_{\{  \varepsilon>\tau_{B(x,h)} \}}\right]\right)^{\frac{\zeta}{1+\zeta}}\\
		\leq &   \left(\E^x\left[\sup_{t\geq 0}(\delta(t)f(X_t))^{1+\zeta} \right]\right)^{\frac{1}{1+\zeta}} \cdot \left(\E^x\left[1_{\{  \varepsilon>\tau_{B(x,h)} \}}\right] \right)^{\frac{\zeta}{1+\zeta}}\\
		\leq &O(1)\cdot 	\left(\P^{x}\left(\tau_{B(x,h)}\leq \eps \right)\right)^{\frac{\zeta}{1+\zeta}},
	\end{aligned} 
	\ee
	where the first inequality follows from $f\geq 0$, the second inequality follows from H\"older's inequality, the third inequality follows from Jensen's inequality, and the last inequality follows from \eqref{eq.assume.wellpose}. Applying Lemma \ref{lm.step.1} with $a=\frac{1+\zeta}{\zeta}$ to \eqref{eq.lm.V}, we have  
	\be\label{eq.lm.epsilon}
	\E^x \left[V(\eps, X_\eps, S)\cdot 1_{\{  \varepsilon>\tau_{B(x,h)} \}} \right]=o(\varepsilon).
	\ee 
	Similarly, we can show that
	$$\E^x \left[V(\varepsilon\wedge \tau_{B(x,h)},X_{ \varepsilon\wedge\tau_{B(x,h)}},S)\cdot 1_{\{  \varepsilon> \tau_{B(x,h)} \}}\right]=o(\varepsilon).$$
	This together with \eqref{eq.lm.epsilon} implies \eqref{eq.lm.vepsilonh}.
\end{proof}

Recall that $L_t^x$ is the local time of $X$ at position $x$ up to time $t$. We have the following result.

\begin{Lemma}\label{lm.local.time}
	Let Assumption \ref{assume.x.elliptic} hold. Then for any $x\in\X$ and $h>0$,
	\begin{equation}\label{eq.local.hepsilon} 
	\lim_{\eps\searrow 0}\frac{\E^x[L^x_{\varepsilon\wedge\tau_{B(x,h)}}]}{\sqrt{\eps}}=\sqrt{\frac{2}{\pi}}\cdot |\sigma(x)|.
	\end{equation}
\end{Lemma}

\begin{proof}
Let $h>0$ and $x\in \X$. Thanks to Assumption \ref{assume.x.elliptic}(i) and \eqref{e10'} (with $m=\frac{p+2}{2}>1$), it holds for any $p,t>0$ that
$$\E^x\left[\sup_{0\leq s\leq t}|X_s|^p\right]\leq 1+\E^x\left[\sup_{0\leq s\leq t}|X_s|^{p+2}\right]<\infty.$$ 
This enables us to apply an argument similar to the proof of Lemma \ref{lm.varepsilon.h} and get that
\be\label{eq.lm.xhepsilon}  
\E^x\left[|X_\varepsilon-x|\right]+o(\eps)=\E^x\left[|X_{\varepsilon\wedge \tau_{B(x,h)}}-x|\right].
\ee 
Applying Lemma \ref{lm.peskir.ito} on $\left[0, \varepsilon\wedge \tau_{B(x,h)}\right]$ with $g(t,y):=|y-x|$ and then taking expectation, and using \eqref{eq.lm.xhepsilon}, we have that
\begin{equation}\label{e31}
	\E^x[|X_\varepsilon-x|]+o(\eps)=\E^x\left[\int_0^{\varepsilon\wedge \tau_{B(x,h)}}\text{sgn}(X_s-x)\mu(X_s)ds\right]+\E^x\left[L^x_{\varepsilon\wedge \tau_{B(x,h)}}\right].
\end{equation}
By Assumption \ref{assume.x.elliptic}(i), the first term on the RHS of \eqref{e31} can be estimated as follows,
\begin{equation}\label{e32}
\left|\E^x\left[\int_0^{\varepsilon\wedge \tau_{B(x,h)}}\text{sgn}(X_s-x)\mu(X_s)ds\right]\right|\leq\sup_{y\in \overline{B(x,h)}}|\mu(y)|\eps=O(\eps).
\end{equation}
	As for the left-hand-side (LHS) of \eqref{e31}, by Lemma \ref{lm.step.2} we have that
	\begin{align}
	\notag \E^x[|X_\eps-x|]&=\E^x[|\widetilde{X}_\eps-x|]+O(\eps)=\E^x[|\mu(x)\eps+\sigma(x)W_\eps|]+O(\eps)\\
	\label{e33}&=|\sigma(x)|\E[|W_\eps|]+O(\eps)=|\sigma(x)|\E[|W_1|]\sqrt{\eps}+O(\eps)=|\sigma(x)|\sqrt{\frac{2}{\pi}}\sqrt{\eps}+O(\eps).
	\end{align} 
	Then \eqref{eq.local.hepsilon} follows from plugging \eqref{e32} and \eqref{e33} into \eqref{e31}.
	\end{proof}

\begin{Lemma}\label{lm.v.lrderiv}
	Let Assumptions \ref{assume.x.elliptic}, \ref{assume.delta.deriv}, \ref{assume.new} hold. Let $S\in\Bc$ and $x\in \partial S$, and suppose  $(x-h,x)\subset S^c$ (resp. $(x,x+h)\subset S^c$) for some $h>0$. Then
	\be 
	V_x(t,x-,S)\leq \delta(t) V_x(0,x-,S)\quad (\text{resp.}\; V_x(t,x+,S)\geq \delta(t) V_x(0,x+,S)).
	\ee
\end{Lemma}

\begin{proof}
Notice that Assumption \ref{assume.new} gives the existence of $V_x(t,y-,S)$ for $y\in(x-h,x]$ (resp. $V_x(t,y+,S)$ for $y\in[x,x+h)$) when $(x-h,x)\subset S^c$ (resp. when $(x,x+h)\subset S^c$). For any $y\in \X, t\geq 0$, by the non-negativity of $f$ and \eqref{eq.assume.DI}, 
	\begin{align*} 
		V(t,y,S)=\E^y[\delta(t+\rho_S)f(X_{\rho_S})]\geq \delta(t)\E^y[\delta(\rho_S)f(X_{\rho_S})]=\delta(t)V(0,y,S).
	\end{align*}
	Suppose $(x-h,x)\subset S^c$. Then by the fact that $V(t,x,S)=\delta(t)f(x)$ (due to Lemma \ref{lm.v.c12} (a)) and the above inequality, we have that
	\begin{align*}
		V_x(t,x-,S)=&\lim_{\varepsilon\searrow 0} \frac{1}{\varepsilon}\left( V(t,x,S)-V(t,x-\varepsilon,S) \right)\\
		\leq &\lim_{\varepsilon\searrow 0} \frac{1}{\varepsilon}\Big(  \delta(t)(f(x)-V(0,x-\varepsilon,S)) \Big)=\delta(t) V_x(0,x-,S).
	\end{align*}  
	Similar argument is applied for the result of $V_x(t,x+,S)$ when $(x,x+h)\subset S^c$.
\end{proof}

Lemmas \ref{lm.local.time} and \ref{lm.v.lrderiv} together indicate that, as long as $V_x(0,x+,S)-V_x(0,x-,S)\neq 0$, the local time integral is the dominating term on the RHS of \eqref{eq.explain.31.2}. Thus, to make the LHS of  \eqref{eq.explain.31.2} non-positive in the limit, $V_x(0,x+,S)-V_x(0,x-,S)$ shall be non-positive. 
Based on this and recalling \eqref{eq.explain.31.0}, we now prove the necessary conditions for a weak equilibrium in the following proposition. The sufficiency part follows next.

\begin{Proposition}\label{prop.inside.cha}
	Let Assumptions \ref{assume.x.elliptic}--\ref{assume.new} hold. Suppose $S$ is an admissible stopping policy. If $S$ is a weak equilibrium, then
	\be\label{eq.prop.weaknoboundary}
	\begin{aligned}
	\begin{cases}
	V_x(0,x+,S)\leq V_x(0,x-,S) & \forall x\in S,\\
\Lc V(0,x+,S)\vee \Lc V(0,x-,S)\leq 0 &  \forall x\in \X.
\end{cases}
	\end{aligned}
	\ee
\end{Proposition}

\begin{proof}
We verify the first inequality in \eqref{eq.prop.weaknoboundary} by contradiction. Take $x\in S$ and suppose 
	\be\label{eq.lm.a} 
	a:=V_x(0,x+,S)-V_x(0,x-,S)>0.
\ee 
	 Recall $\Gc$ defined in \eqref{eq.assume.G}. Choose $h>0$ such that $(x-h,x)\cup(x,x+h)$ is contained in $(\Gc\cap S^\circ)\cup S^c$. By Lemma \ref{lm.v.c12}(a) and Assumption \ref{assume.f.positive}(ii), $V\in \Cc^{1,2}([0,\infty)\times (x-h,x])$ and  $V\in \Cc^{1,2}([0,\infty)\times [x,x+h))$. Then we can  
	apply Lemma \ref{lm.peskir.ito} to get
	\be\label{eq.ito.tanaka}
	\begin{aligned}
		V(\varepsilon\wedge \tau_{B(x,h)}, X_{\varepsilon\wedge \tau_{B(x,h)}},S)-V(0,x,S)=&\int_0^{\varepsilon\wedge \tau_{B(x,h)}} \frac{1}{2}(\Lc V(s,X_{s}-,S)+\Lc V(s, X_s+,S))ds\\
		&+\int_0^{\varepsilon\wedge \tau_{B(x,h)}} V_x(s, X_s,S)\sigma(X_s)\cdot 1_{\{X_s\neq x \}}dW_s\\
		&+\frac{1}{2}\int_0^{\varepsilon\wedge \tau_{B(x,h)}}(V_x(s,x+,S)-V_x(s,x-,S))dL^x_s.
	\end{aligned}
	\ee
	Let $\eps\in(0,1)$, notice that the diffusion integrand above is bounded on $[0,1]\times\overline{B(x,h)}$. Taking expectation on both sides of \eqref{eq.ito.tanaka} and then applying Lemma \ref{lm.varepsilon.h}, we have that
	\be\label{eq.ito.exptanaka}  
	\begin{aligned}
		\E^x[V(\varepsilon, X_{\varepsilon},S)-V(0,x,S)]=&\E^x\left[\int_0^{\varepsilon\wedge \tau_{B(x,h)}} \frac{1}{2}(\Lc V(s,X_{s}-,S)+\Lc V(s, X_s+,S))ds\right]\\
		&+\E^x\left[\frac{1}{2}\int_0^{\varepsilon\wedge \tau_{B(x,h)}}(V_x(s,x+,S)-V_x(s,x-,S))dL^x_s\right]+o(\varepsilon).
	\end{aligned}
	\ee
	By Lemma \ref{lm.v.lrderiv} and \eqref{eq.lm.a},
	$$
	V_x(t,x+,S)-V_x(t,x-,S)\geq \delta(t)(V_x(0,x+,S)-V_x(0,x-,S))=a\delta(t), \quad \forall t\geq 0.
	$$ 
	By the above inequality and the continuity of $\delta$, we can take $T>0$ such that 
	\bee
	V_x(s,x+,S)-V_x(s,x-,S)\geq \frac{a}{2},\quad\forall s\in[0,T].
	\eee
	Then for $\eps\in[0, T\wedge 1]$, the second term on the RHS of \eqref{eq.ito.exptanaka} can be estimated as follows
	\be\label{eq.prop.vderiv} 
	\E^x\left[\frac{1}{2}\int_0^{\varepsilon\wedge \tau_{B(x,h)}}(V_x(s,x+,S)-V_x(s,x-,S))dL^x_s\right]\geq \frac{a}{4}\E^x[L^x_{\tau_{B(x,h)}\wedge \varepsilon}].
	\ee 
	By Lemma \ref{lm.v.c12}(b), we have 
	$$
	\sup_{(t,y)\in [0,1]\times \overline{B(x,h)}} |\Lc V(t,y-,S)+\Lc V(t, y+,S)|<\infty,
	$$ 
	and thus the first term on the RHS of \eqref{eq.ito.exptanaka} is of order $O(\varepsilon)$. Plugging this and \eqref{eq.prop.vderiv} into \eqref{eq.ito.exptanaka} and then applying Lemma \ref{lm.local.time}, we have
$$
		\liminf_{\varepsilon\searrow 0}\frac{1}{\varepsilon} \E^x[V(\varepsilon, X_{\varepsilon},S)-V(0,x,S)]
		\geq  O(1)+\frac{a}{4}\liminf_{\varepsilon\searrow 0}\frac{1}{\varepsilon}\E^x[L^x_{\tau_{B(x,h)}\wedge \varepsilon}]
		=\infty,
$$
	which contradicts $S$ being a weak equilibrium. Hence, $V_x(0,x+,S)-V_x(0,x-,S)\leq 0$.
	
	Next, we verify the second inequality in \eqref{eq.prop.weaknoboundary}. Take $x\in\X$ and we consider three cases.
	
	Case (i) $x\in S^c$. Lemma \ref{lm.v.c12}(a) shows that $\Lc V(0,x,S)=0$.
	
	Case (ii) $x\in \Gc\cap S^\circ$. Choose $h>0$ such that $B(x,h)\subset \Gc\cap S^\circ$. Notice that $V(t,y,S)=\delta(t)f(y)$ for $y\in S^\circ$. Then by Assumptions \ref{assume.x.elliptic}(i) and \ref{assume.f.positive}(ii), we have $V(t,y,S)\in \Cc^{1,2}([0,\infty)\times \overline{B(x,h)})$ and $(t,y)\mapsto \Lc V(t,y,S)$ is continuous on $[0,\infty)\times \overline{B(x,h)}$. Thus,
	$$
	\lim_{\varepsilon\rightarrow 0}\frac{1}{\varepsilon}\int_0^{\varepsilon\wedge \tau_{B(x,h)}}\Lc V(s, X_s,S)ds =\Lc V(0,x,S), \; \P^{x}-\text{a.s.}.$$
	By Lemma \ref{lm.v.c12}(b), $\sup_{(t,y)\in [0,\infty)\times \overline{B(x,h)}} |\Lc V(t,y,S)|<\infty$. Then we can apply the dominated convergence theorem to derive
	\be\label{eq.generator.lv0}    
	\lim_{\varepsilon\rightarrow 0}\frac{1}{\varepsilon}\E^x\left[\int_0^{\varepsilon\wedge \tau_{B(x,h)}}\Lc V(s, X_s,S)ds \right]=\Lc V(0,x,S).
	\ee	
	Notice that \eqref{eq.ito.exptanaka} is valid and the local time integral term in \eqref{eq.ito.exptanaka} vanishes in this case. Then \eqref{eq.ito.exptanaka} and \eqref{eq.generator.lv0} together lead to
	\bee
	\begin{aligned}
		\lim_{\varepsilon\searrow 0}\frac{1}{\varepsilon}\E^x\left[V(\varepsilon, X_{\varepsilon},S)-V(0,x,S)\right]=\lim_{\varepsilon\searrow 0}\frac{1}{\varepsilon}\E^x\left[\int_0^{\varepsilon\wedge \tau_{B(x,h)}} \Lc V(s,X_{s},S)ds\right]=\Lc V(0,x,S).
	\end{aligned}
	\eee
Since $S$ is a weak equilibrium, we have $\Lc V(0,x,S)\leq 0$.
	
	{Case (iii) $x\in S\setminus (\Gc\cap S^\circ)$. As $S$ is admissible, we can pick $h>0$ such that $(x-h,x)$ is contained in either $\Gc\cap S^\circ$ or $S^c$. By the results in Cases (i) and (ii), as well as the continuity of $x\mapsto \Lc V(0,x-,S)$ on $(x-h,x]$, we have that}
		$$\Lc V(0,x-,S)=\lim_{\varepsilon\searrow 0}  \Lc V(0,(x-\varepsilon)-,S)\leq 0.$$
		Similarly, $\Lc V(0,x+,S)\leq 0$.
\end{proof}

\begin{proof}[{\bf Proof of Theorem \ref{prop.weak.cha}}]
The necessity is implied by Proposition \ref{prop.inside.cha}. Let us prove the sufficiency.

	Take $x\in S$. Since $S$ is admissible, by Lemma \ref{lm.v.c12} and Assumption  \ref{assume.f.positive}(ii), no matter $x\in S^\circ$ or $x\in \partial S$, we can choose $h>0$ such that $V(t,x,S)\in \Cc^{1,2}([0,\infty)\times (x-h, x])$ and $V(t,x,S)\in \Cc^{1,2}([0,\infty)\times [x, x+h))$. By a similar argument as that for \eqref{eq.ito.exptanaka} (with  Lemmas \ref{lm.peskir.ito} and \ref{lm.varepsilon.h} applied), we have that
	\be\label{eq.epsilon.tanaka}  
\begin{aligned}
\frac{1}{\varepsilon}(\E^x[V(\varepsilon, X_{\varepsilon},S)]-V(0,x,S))=&\frac{1}{\varepsilon}\E^x\left[\int_0^{\varepsilon\wedge \tau_{B(x,h)}} \frac{1}{2}(\Lc V(s,X_{s}-,S)+\Lc V(s, X_s+,S))ds\right]\\
	&+\frac{1}{\varepsilon}\E^x\left[\frac{1}{2}\int_0^{\varepsilon\wedge \tau_{B(x,h)}}(V_x(s,x+,S)-V_x(s,x-,S))dL^x_s\right]+o(1)
\end{aligned}
\ee	
By \eqref{eq.weak.equicha2} and the (left/right) continuity of $(s,y)\mapsto\Lc V(s,y\pm, S)$ at $(0,x)$, for $\P$-a.s. $\omega\in \Omega$,
$$
\limsup_{s \searrow 0} \frac{1}{2}(\Lc V(s,X_{s}(\omega)-,S)+\Lc V(s, X_s(\omega)+,S))\leq 0,
$$
which leads to
\be\label{eq.lv.0} 
\limsup_{\varepsilon \searrow 0}\frac{1}{\varepsilon}\int_0^{\varepsilon\wedge \tau_{B(x,h)}} \frac{1}{2}(\Lc V(s,X_{s}(\omega)-,S)+\Lc V(s, X_s(\omega)+,S))ds\leq 0.
\ee
By Lemma \ref{lm.v.c12} (b),
$$
\sup_{(t,y)\in [0,1]\times \overline{B(x,h)}} |\Lc V(t,y-,S)+\Lc V(t, y+,S)|<\infty.
$$
This enables us to apply Fatou's lemma for \eqref{eq.lv.0} and get
\be\label{eq.weak.lv} 
\limsup_{\varepsilon \searrow 0}\frac{1}{\varepsilon}\E^x\left[\int_0^{\varepsilon\wedge \tau_{B(x,h)}} \frac{1}{2}(\Lc V(s,X_{s}-,S)+\Lc V(s, X_s+,S))ds\right]\leq 0.
\ee 
By \eqref{eq.vx.lineart},
$$
V_x(t,x+,S)-V_x(t,x-,S)\leq V_x(0,x+,S)-V_x(0,x-,S)+o(\sqrt{t}).
$$
This together with \eqref{eq.weak.equicha1} implies that
\be\label{eq.weak.local}
\begin{aligned}
&\frac{1}{\varepsilon}\E^x\left[\frac{1}{2}\int_0^{\varepsilon\wedge \tau_{B(x,h)}}(V_x(s,x+,S)-V_x(s,x-,S))dL^x_s\right]\\
&\leq \frac{1}{\varepsilon}\E^x\left[\frac{1}{2}\int_0^{\varepsilon\wedge \tau_{B(x,h)}}(V_x(0,x+,S)-V_x(0,x-,S)+o(\sqrt{\eps}))dL^x_s\right]\\
&\leq  \frac{1}{2\varepsilon}\cdot o(\sqrt{\eps})\cdot \E^x\left[L^x_{\varepsilon\wedge \tau_{B(x,h)}}\right]\\
&=\frac{1}{2\varepsilon}\cdot o(\sqrt{\eps})\cdot O(\sqrt{\eps})=o(1),
\end{aligned}
\ee
where the last line follows from Lemma \ref{lm.local.time}.
Then by \eqref{eq.epsilon.tanaka}, \eqref{eq.weak.lv} and \eqref{eq.weak.local} we have that
$$\limsup_{\eps\searrow 0}\frac{1}{\varepsilon}(\E^x[V(\varepsilon, X_{\varepsilon},S)]-V(0,x,S))\leq 0.$$
\end{proof}

\section{Optimal Mild Equilibria are Weak Equilibria}\label{sec:ws.optimal}

In this section, we show that an optimal mild equilibrium is a weak equilibrium.

To begin with, let us point out that optimal mild equilibria exist for one-dimensional diffusions. Such existence result is provided in \cite[Theorem 4.12]{MR4116459}, and we summarize it in the current context as follows.

\begin{Lemma}\label{lm.optimal.mild}
Let Assumptions \ref{assume.x.elliptic}, \ref{assume.delta.deriv}, \eqref{eq.assume.sstar'} and \eqref{eq.assume.wellpose1} hold. Then
	\begin{equation}\label{eq.nota.sstar}
		S^*:=\cap_{S\in\mathcal{E}}S
	\end{equation}
	is an optimal mild equilibrium, where $\mathcal{E}$ is the set containing all mild equilibria. 
\end{Lemma}

\begin{Remark}
Notice that \eqref{e2} is assumed in \cite[Theorem 4.12]{MR4116459} for $S^*$ being an optimal mild equilibrium, which is guaranteed by Assumption \ref{assume.x.elliptic} as stated in Remark \ref{rm.close.regular}. Also, \eqref{eq.assume.wellpose1} can be deduced from Assumption \ref{assume.f.positive}(i) as stated in Remark \ref{rm.f.wellpose}.
\end{Remark}

Below is the main result of this section. 

\begin{Theorem}\label{thm.optimalmild.weak}
Let Assumptions \ref{assume.x.elliptic}--\ref{assume.new} hold and $S$ be an admissible stopping policy. If $S$ is an optimal mild equilibrium, then it is also a weak equilibrium.
\end{Theorem}

	By Lemma \ref{lm.optimal.mild} and Theorem \ref{thm.optimalmild.weak} we have the following.
\begin{Corollary}
	Let Assumptions \ref{assume.x.elliptic}--\ref{assume.new} hold. Suppose $S^*$ is admissible. Then $S^*$ is also weak.
\end{Corollary}

\begin{Remark}\label{rm.sstar.best}
	The above corollary also provides the existence of weak equilibria (ignoring admissibility) as a by-product. Moreover, since any weak equilibrium is also mild, we can see that $S^*$ is optimal among all mild and weak equilibria.
\end{Remark}

\subsection{Proof of Theorem \ref{thm.optimalmild.weak}}

	To show an optimal mild equilibrium $S$ is a weak equilibrium, by Theorem \ref{prop.weak.cha} it suffices to verify \eqref{eq.weak.equicha1} and \eqref{eq.weak.equicha2} for $S$. \eqref{eq.weak.equicha1} will be proved in Proposition \ref{prop.weak.optimalderiv} by contradiction. In particular, if we assume $V_x(0,x_0+,S)-V_x(0,x_0-,S)>0$, then a mild equilibrium {\it better} than $S$ can be construct by ``digging a small hole $B(x_0,h)$" out of $S$. The proof of \eqref{eq.weak.equicha2} is also carried out via contradiction by finding a  {\it better} mild equilibrium.
	
Such construction of a  {\it better} mild equilibrium requires the comparison between the expectation of a local time integral before the exit time $\tau_{B(x,h)}$ and the expectation of $\tau_{B(x,h)}$ for small $h$, which is stated in the following lemma.

\begin{Lemma}\label{lm.hit.localh}
Suppose Assumptions \ref{assume.x.elliptic} and \ref{assume.delta.deriv} hold. For $x_0\in \X$,  we have that
\begin{align}
\label{eq.hit.local2}  \dfrac{\E^{x_0+rh}[\int_0^{\tau_{B(x_0, h)}}\delta(t) dL^{x_0}_t]\cdot h}{\E^{x_0+rh}[\tau_{B(x_0, h)}]} \overset{h\searrow 0}{\longrightarrow} \dfrac{\sigma^2(x_0)}{1+|r|}\quad \text{uniformly for $r\in (-1,1)$}.
\end{align}
\end{Lemma}

\begin{proof}
We first prove
\be\label{eq.hit.local1}  
\dfrac{\E^{x_0+rh}[\tau_{B(x_0, h)}]}{(1-r^2)h^2}\overset{h\searrow 0}{\longrightarrow}\dfrac{1}{\sigma^2(x_0)}\quad  \text{uniformly for $r\in (-1,1)$}
\ee
by using an argument similar to that for \cite[Lemma A.5]{MR4080735}.
 Pick a constant $a$ and consider the function 
$g(t,z):=a(z-x_0)^2-t.$
We have that
$$\Lc g(t,z)=-1+a \sigma^2(z)+\mu(z)2a(z-x_0).$$
By Assumption \ref{assume.x.elliptic}(ii), $\sigma^2(x_0)>0$. For any constant $a>\frac{1}{\sigma^2(x_0)}$, by the continuity of $\mu(x)$ and $\sigma(x)$, we can find $h>0$, which only depends on $a$, such that $\Lc g(t,z)\geq 0$ for any $z\in B(x_0,h)$. Applying Ito's formula to $g(t,X_t)$, we have that
$$
a\E^y\left[(X_{\tau_{B(x_0, h)}}-x_0)^2\right]-\E^y[\tau_{B(x_0, h)}]-a(y-x_0)^2=\E^y\left[\int_0^{\tau_{B(x_0, h)}}\Lc g(t,X_t)dt\right]\geq 0,\quad \forall y\in B(x_0,h).
$$
For $y\in B(x_0,h)$, rewrite $y=x_0+rh$ for some $r\in (-1,1)$. Then the above inequality leads to
\be\label{eq.hit.rh0}  
\E^{x_0+rh}[\tau_{B(x_0,h)}]\leq a\E^{x_0+rh}\left[(X_{\tau_{B(x_0,h)}}-x_0)^2\right]-a(rh)^2= ah^2-ar^2h^2=a(1-r^2)h^2,\quad \forall r\in (-1,1).
\ee
Similarly, for any constant $0<\tilde{a}<\frac{1}{\sigma^2(x_0)}$, we can find $\tilde{h}$ which only depends on $\tilde{a}$, such that $\Lc g(t,z)\leq 0$ on $B(x_0, \tilde{h})$, and
\be\label{eq.hit.rh1} 
\E^{x_0+r\tilde{h}}[\tau_{B(x_0,\tilde{h})}]\geq \tilde{a}(1-r^2)\tilde{h}^2,\quad \forall r\in (-1,1).
\ee
By \eqref{eq.hit.rh0} and \eqref{eq.hit.rh1}, for any $H\in (0,h\wedge \tilde{h}]$ we have that
$$
\tilde{a}\leq \frac{\E^{x_0+rH}[\tau_{B(x_0,H)}]}{(1-r^2)H^2}\leq a,\quad \text{for all } r\in (-1,1),\ \tilde{a}\in\left(0,\frac{1}{\sigma^2(x_0)}\right)\; \text{and}\; a>\frac{1}{\sigma^2(x_0)}.
$$    
Let $\tilde{a}=\frac{1}{\sigma^2(x_0)}-\varepsilon$ and $a=\frac{1}{\sigma^2(x_0)}+\varepsilon$ for any $\varepsilon>0$ and then take $H\searrow 0$ for the above inequality. By the arbitrariness of $\eps$, \eqref{eq.hit.local1} follows.

Next, we prove \eqref{eq.hit.local2}. Consider the function
$\mathfrak{g}(t,z):=\delta(t)|z-x_0|.$
For $h>0$ and $y\in B(x_0,h)$, applying Lemma \ref{lm.peskir.ito} to $\mathfrak{g}(t, X_t)$ we have that
\be\label{eq.int.deltaL}  
\begin{aligned}
	\E^{y}[\delta(\tau_{B(x_0, h)})|X_{\tau_{B(x_0, h)}}-x_0|]-|y-x_0|=&\E^{y}\left[\int_0^{\tau_{B(x_0, h)}} \frac{1}{2}(\Lc \mathfrak{g}(t,X_t-)+\Lc \mathfrak{g}(t,X_t+))dt\right]\\
	&+\E^{y}\left[\int_0^{\tau_{B(x_0, h)}} \frac{1}{2}(1-(-1))\delta(t) dL^{x_0}_t\right].
\end{aligned}
\ee
By Lemma \ref{lm.delta.inequ}, $|\delta'(t)|\leq |\delta'(0)|\delta(t)\leq |\delta'(0)|$. This implies that 
\be\label{eq.int.deltaL0}    
\begin{aligned}
	\left|\frac{1}{2}\big(\Lc \mathfrak{g}(t,z-)+\Lc  \mathfrak{g}(t,z+)\big)\right|\leq& |\delta'(t)|\cdot |z-x_0|+ \delta(t)|\mu(z)|\leq |\delta'(0)|\cdot |z-x_0|+|\mu(z)|.
\end{aligned}
\ee
By \eqref{eq.int.deltaL}  and \eqref{eq.int.deltaL0}, we have that
\be\label{e12}
\begin{aligned}
	&h\E^y[\delta(\tau_{B(x_0, h)})]-|y-x_0|-\E^y\left[\int_0^{\tau_{B(x_0, h)}}(|\delta'(0)||X_t-x_0|+|\mu(X_t)|)dt\right]\\
	&\leq  \E^y\left[\int_0^{\tau_{B(x_0, h)}}\delta(t) dL^{x_0}_t\right]\\
	& \leq  h\E^y[\delta(\tau_{B(x_0, h)})]-|y-x_0|+\E^y\left[\int_0^{\tau_{B(x_0, h)}}(|\delta'(0)||X_t-x_0|+|\mu(X_t)|)dt\right].
\end{aligned}
\ee
Notice that $|X_t-x|\leq h$ for $t\leq\tau_{B(x_0, h)}$ and $\sup_{z\in B(x_0,1)}|\mu(z)|\leq K$ for some constant $K>0$ that depends on $x_0$. Then for $h\leq 1$, by rewriting $y=x_0+rh$ in \eqref{e12} we have that
\bee
\begin{aligned}
	&h\E^{x_0+rh}[\delta(\tau_{B(x_0, h)})]-h|r|-(h|\delta'(0)|+K)\cdot \E^{x_0+rh}[\tau_{B(x_0, h)}]\\
	& \leq \E^{x_0+rh}\left[\int_0^{\tau_{B(x_0, h)}}\delta(t) dL^{x_0}_t\right]\\
	& \leq h\E^{x_0+rh}[\delta(\tau_{B(x_0, h)})]-h|r|+(h|\delta'(0)|+K)\cdot \E^{x_0+rh}[\tau_{B(x_0, h)}], \quad \forall r\in (-1,1).
\end{aligned}
\eee 
Then
\be\label{eq.hit.local0} 
\begin{aligned}
	&\dfrac{h^2(\E^{x_0+rh}[\delta(\tau_{B(x_0, h)})]-|r|)}{\E^{x_0+rh}[\tau_{B(x_0,h)}]}-\dfrac{(|\delta'(0)|h^2+Kh)\cdot \E^{x_0+rh}[\tau_{B(x_0, h)}]}{\E^{x_0+rh}[\tau_{B(x_0,h)}]}\\
	&\leq \Big(h\E^{x_0+rh}\left[\int_0^{\tau_{B(x_0, h)}}\delta(t) dL^{x_0}_t\right]\Big)/\Big(\E^{x_0+rh}[\tau_{B(x_0, h)}]\Big)\\
	& \leq  \dfrac{h^2(\E^{x_0+rh}[\delta(\tau_{B(x_0, h)})]-|r|)}{\E^{x_0+rh}[\tau_{B(x_0,h)}]}+\dfrac{(|\delta'(0)|h^2+Kh)\cdot \E^{x_0+rh}[\tau_{B(x_0, h)}]}{\E[\tau_{B(x_0,h)}]}, \quad \forall r\in (-1,1).
\end{aligned}
\ee 
By the second inequality in Lemma \ref{lm.delta.inequ}, it holds uniformly in $r\in(-1,1)$ that
\begin{align*}
|\E^{x_0+rh}[\delta(\tau_{B(x_0, h)})]-1|=\E^{x_0+rh}[1-\delta(\tau_{B(x_0, h)})]\leq |\delta'(0)|\cdot\E^{x_0+rh}[\tau_{B(x_0, h)}]\overset{h\searrow 0}{\longrightarrow} 0.
\end{align*}
This together with \eqref{eq.hit.local1} implies that
\begin{equation}\label{e34}
\dfrac{h^2(\E^{x_0+rh}[\delta(\tau_{B(x_0, h)})]-|r|)}{\E^{x_0+rh}[\tau_{B(x_0,h)}]}\overset{h\searrow 0}{\longrightarrow}\frac{\sigma^2(x_0)}{1-r^2}\cdot (1-|r|)=\frac{\sigma^2(x_0)}{1+|r|} \quad\text{uniformly for } r\in (-1,1).
\end{equation}
Notice that $\lim_{h\searrow 0}(|\delta'(0)|h^2+Kh)=0$. This together with \eqref{eq.hit.local0} and \eqref{e34} implies \eqref{eq.hit.local2}.
\end{proof}

Now we are ready to deal with \eqref{eq.weak.equicha1} in the following proposition. The verification for \eqref{eq.weak.equicha2} follows next.

\begin{Proposition}\label{prop.weak.optimalderiv}
Let Assumptions \ref{assume.x.elliptic}--\ref{assume.new} hold and $S$ be an admissible stopping policy. If $S$ is an optimal mild equilibrium, then
$$
V_x(0,x-,S)\geq V_x(0,x+,S)\quad \forall x\in S.
$$
\end{Proposition}

\begin{proof}
Notice that Assumption \ref{assume.f.positive}(ii) and Lemma \ref{lm.v.c12} guarantees the existence of $V_x(t,x\pm,S)$ and $\Lc V(t,x\pm, S)$ for any $(t,x)\in [0,\infty)\times \X$. We prove the desired result by contradiction. Take $x_0\in S$ and suppose
\be\label{eq.prop.a} 
a:=V_x(0,x_0+,S)-V_x(0,x_0-,S)>0.
\ee 
Recall $\Gc$ defined in \eqref{eq.assume.G}. To reach to a contradiction, we will construct a new mild equilibrium, which is strictly better than $S$, for each of the three cases:
(i) $x_0\in\partial S$ for boundary case (\ref{boundary.a});
(ii) $x_0\in \partial S$ for boundary case (\ref{boundary.b});
(iii) $x_0\in S^\circ$.

{\bf Case (i)} $x_0\in\partial S$ for boundary case (\ref{boundary.a}).
Without loss of generality, we assume that $(x_0, x_0+h_0)\subset (S^\circ \cap \Gc)$ and $(x_0-h_0,x_0)\subset S^c$ for some $h_0>0$. Denote
$l:=\sup\{y\leq x_0-h_0: y\in S\},$
and note that $l$ can be $-\infty$. We proceed the proof for this case in three steps.

\textit{Step 1.} We show that there exists $h\in(0,h_0)$ such that,
\be\label{eq.V.h} 
\E^{y}\left[V(\tau_{B(x_0,h)}, X_{\tau_{B(x_0,h)}}, S)\right]-f(y)>0, \quad \forall y\in(x_0-h, x_0+h). 
\ee
Notice from Lemma \ref{lm.v.lrderiv} and \eqref{eq.prop.a} that
\be\label{eq.lr.a} 
V_x(t,x_0+,S)-V_x(t,x_0-,S)\geq \delta(t) (V_x(0,x_0+,S)-V(0,x_0-,S))=a\delta(t),\quad \forall t\geq 0.
\ee
Fix $h\in(0,h_0)$ and pick an arbitrary $x\in B(x_0, h)$. For all $n\in \N$, write
\begin{equation}\label{eq791}
\tau_n:= \tau_{B(x_0,h)}\wedge n
\end{equation}
for short. We apply Lemma \ref{lm.peskir.ito} to $V(t,X_t,S)$ on $[0, \tau_n]$ and take expectation; the diffusion term vanishes under expectation due to Lemma \ref{lm.v.c12} and continuity of $\sigma$. Then combining with \eqref{eq.lr.a}, we have that, 
\be\label{eq.bound.f1deriv'}
\begin{aligned}
	&\E^x\left[V\left(\tau_n, X_{\tau_n},S \right)\right]-V(0,x,S)\\
	&\geq \E^x\left[\int_0^{\tau_n} \frac{1}{2}(\Lc V(t, X_t-,S)+\Lc V(t, X_t+,S))dt \right]
	+a\E^x\left[\int_0^{\tau_n}\delta(t)dL^{x_0}_t\right].
\end{aligned}
\ee  
By Lemma \ref{lm.v.c12}(b), we have 
$$M:=\sup_{(t,y)\in [0,\infty)\times \overline{B(x_0,h_0)}}\frac{1}{2}\left( |\Lc V(t, y-,S)|+|\Lc V(t, y+,S)| \right)<\infty.$$
This together with \eqref{eq.bound.f1deriv'} implies that
\be\label{eq.bound.f1deriv} 
\begin{aligned}
\E^{x}[V(\tau_n, X_{\tau_n},S)]-V(0,{x},S)
\geq -M\E^{x}[\tau_n] +a\E^{x}\left[\int_0^{\tau_n}\delta(t)dL^{x_0}_t\right]\quad \forall n\in \N.
\end{aligned}
\ee 
For the LHS of \eqref{eq.bound.f1deriv}, \eqref{eq.assume.wellpose1} readily implies that
\be\label{eq.prop.4.1'} 
\lim_{n\rightarrow \infty}\E^{x_0}[V(\tau_n, X_{\tau_n},S)=\E^{x_0}[V(\tau_{B(x_0,h)},X_{\tau_{B(x_0,h)}},S)].
\ee
Indeed, set
$
\eta_n:=\inf\{t\geq \tau_n, X_t\in S\}$ for all $n\in \N$, and $\eta:=\inf\{t\geq \tau_{B(x_0,h)}, X_t\in S\}$. We have
$\E^{x}[V(\tau_n,X_{\tau_n},S)]=\E^{x}\left[\E^{x}\left[\delta(\eta_n)f(X_{\eta_n})|\mathcal{F}_{\tau_n}\right]\right]=\E^{x}[\delta(\eta_n)f(X_{\eta_n})]$ for all $n\in \N,$
and $\E^{x_0}[V(\tau_{B(x_0,h)},X_{\tau_{B(x_0,h)}},S)]=\E^{x_0}[\delta(\eta)f(X_\eta)]$. As $n\to\infty$, $\eta_n\to\eta$, $\P^x$-a.s.. 
Then by Assumption \ref{assume.f.positive}, we can apply the dominated convergence theorem to get
$
	 \lim_{n\rightarrow \infty}\E^{x_0}[\delta(\eta_n)f(X_{\eta_n})]
	=\E^{x_0}[\delta(\eta)f(X_\eta)],
$ i.e., \eqref{eq.prop.4.1'} holds. (Note that \eqref{eq.assume.sstar'} is used on $\{\eta=\infty\}$.)

Applying the monotone convergence theorem to the RHS of \eqref{eq.bound.f1deriv} and combining with \eqref{eq.prop.4.1'}, we have that
\be\label{e13'}	 
\begin{aligned}
	\E^{x}[V(\tau_{B(x_0, h)}, X_{B(x_0, h)},S)]-V(0,x,S)
\geq -M\E^{x}[\tau_{B(x_0, h)}] +a\E^{x}\left[\int_0^{\tau_{B(x_0, h)}}\delta(t)dL^{x_0}_t\right].
\end{aligned}
\ee 
By the arbitrariness of $x\in B(x_0,h)$, 
\be\label{e13}	 
\begin{aligned}
	&\E^{x_0+rh}[V(\tau_{B(x_0, h)}, X_{B(x_0, h)},S)]-V(0,{x_0+rh},S)\\
	&\geq -M\E^{x_0+rh}[\tau_{B(x_0, h)}] +a\E^{x_0+rh}\left[\int_0^{\tau_{B(x_0, h)}}\delta(t)dL^{x_0}_t\right],
	\quad \forall r\in (-1,1).
\end{aligned}
\ee
By Lemma \ref{lm.hit.localh} and $|\sigma(x_0)|>0$, we can choose the above $h$ small enough such that
$$\E^{x_0+rh}\left[\int_0^{\tau_{B(x_0,h)}}\delta(t)dL^{x_0}_t\right]\geq \left(\frac{M}{a}+1\right)\E^{x_0+r{h}}[\tau_{B(x_0,h)}], \quad\forall r\in(-1,1).$$
Consequently, \eqref{e13} leads to
$$
\E^{x_0+rh}[V(\tau_{B(x_0,h)}, X_{\tau_{B(x_0,h)}},S)]-V(0,{x_0+rh},S)
\geq a\E^{x_0+rh}[\tau_{B(x_0,h)}] >0,
\quad \forall r\in (-1,1),
$$
which gives \eqref{eq.V.h}.

\textit{Step 2.} 
In the rest part of Case (i), we take $h$ such that \eqref{eq.V.h} holds and write $S_h:= S\setminus B(x_0, h)$ for short. In this step, we prove by contradiction that
\be\label{eq.J.Vh}   
J(y, S_h)\geq  \E^{y}[V(\tau_{B(x_0,h)}, X_{\tau_{B(x_0,h)}}, S)], \quad \forall y\in[x_0,x_0+h).
\ee 
Suppose 
\be\label{eq.step2.contra} 
 \alpha:=\inf\limits_{y\in [x_0, x_0+h]} \Big(J(y,S_h)-\E^y[V(\tau_{B(x_0,h)}, X_{\tau_{B(x_0,h)}}, S)] \Big)<0.
\ee 
As $x_0+h\in S^\circ$, by Lemma \ref{lm.v.c12}(a),
$$J(x_0+h,S_h)=f(x_0+h)=\E^{x_0+h}[V(\tau_{B(x_0,h)}, X_{\tau_{B(x_0,h)}}, S)].$$
By the continuity of functions $y\mapsto J(y,S_h)$ and $y\mapsto\E^y[V(\tau_{B(x_0,h)}, X_{\tau_{B(x_0,h)}}, S)]$ on $[x_0,x_0+h]$, there exists $z^*\in[x_0,x_0+h)$ such that the infimum in \eqref{eq.step2.contra} is attained at $z^*$, i.e.,
\be\label{eq.contra.z} 
J(z^*,S_h)-\E^{z^*}[V(\tau_{B(x_0,h)}, X_{\tau_{B(x_0,h)}}, S)]=\alpha.
\ee 
Define 
$$\nu:=\inf\{t\geq\tau_{B(x_0,h)}:\,X_t\in S\}\quad\text{and }\quad A:=\{X_{\tau_{B(x_0,h)}}=x_0-h,\ X_\nu=x_0,\ \nu<\infty\}.$$
Notice that $\rho_{S_h}=\nu$, $\P^{z^*}$-a.s. on both sets $\{X_{\tau_{B(x_0,h)}}=x_0+h\}$ and $\{X_{\tau_{B(x_0,h)}}=x_0-h, X_\nu <x_0 \}$.
We have that
\be\notag
\begin{aligned}
	J({z^*}, S_h)-\E^{z^*}[V(\tau_{B(x_0,h)}, X_{\tau_{B(x_0,h)}}, S)] &= \E^{z^*}\left[1_A\cdot\big(\delta(\rho_{S_h})f(X_{\rho_{S_h}})-\delta(\nu) f(X_{\nu})\big)\right]\\
	&\geq \E^{z^*}\left[1_A\delta(\nu)\cdot\Big(\E^{z^*}[\delta(\rho_{S_h}-\nu) f(X_{\rho_{S_h}}) \mid \Fc_{\nu}]- f(X_{\nu}) \Big)\right]\\
&=\E^{z^*}[1_A\delta(\nu)]\cdot \big(J(x_0,S_h)- f(x_0) \big)\\
&>\E^{z^*}[1_A\delta(\nu)]\cdot \left(J(x_0,S_h)- \E^{x_0}[V(\tau_{B(x_0,h)}, X_{\tau_{B(x_0,h)}}, S)]\right)\\
&\geq\E^{z^*}[1_A\delta(\nu)]\cdot\alpha\\
&>\alpha,
\end{aligned}
\ee 
where the second (in)equality follows from \eqref{eq.assume.DI} and $f\geq 0$, the third (in)equality follows from the strong Markov property of $X$ and the fact that $X_v=x_0$ on $A$, the fourth (in)equality follows from \eqref{eq.V.h} with $y=x_0$, the fifth (in)equality follows from the definition of $\alpha$ in \eqref{eq.step2.contra}, and the last (in)equality follows from the fact that $\nu\geq\tau_{B(x_0,h)}>0$ and $\delta(t)<1$ for $t>0$. This contradicts \eqref{eq.contra.z}. Therefore, \eqref{eq.J.Vh} holds.

\textit{Step 3.} Now we prove that $S_h$ is a mild equilibrium and is strictly better than $S$. By \eqref{eq.V.h} and \eqref{eq.J.Vh} and noticing that $(x_0,x_0+h)\subset S^\circ$, we have
\be\label{eq.f.jyh}
J(y,S_h)>f(y)=J(y,S),\quad \forall y\in [x_0,x_0+h).
\ee
Then for any $y\in (l,x_0)$, we have that
\be\label{eq.casei.step3.0} 
\begin{aligned}
J(y,S_h)-J(y,S)&=\E^y\left[1_{\{X_{\rho_S}=x_0,\, \rho_S<\infty \}}\left(\delta(\rho_{S_h})f(X_{\rho_{S_h}})-\delta(\rho_S)f(x_0)\right)\right]\\
&\geq\E^y\left[1_{\{X_{\rho_S}=x_0,\, \rho_S<\infty \}}\delta(\rho_S)\left(\E^y\big[\delta(\rho_{S_h}-\rho_S)f(X_{\rho_{S_h}})\big|\mathcal{F}_{\rho_S}\big]-f(x_0)\right)\right]\\
&=\E^y\left[1_{\{X_{\rho_S}=x_0,\, \rho_S<\infty \}}\delta(\rho_S)\big(J(x_0,S_h)-f(x_0)\big)\right]\\
&\geq 0.
\end{aligned}
\ee 
where the second (in)equality follows again from \eqref{eq.assume.DI} and the non-negativity of $f$, the third (in)equality follows from the strong Markov property of $X$, and the last (in)equality follows from \eqref{eq.f.jyh} with $y=x_0$. As $S$ is a mild equilibrium, above inequality implies
\be\label{eq.casei.step3.1} 
J(y, S_h)\geq J(y,S)\geq f(y),\quad \forall y\in(l,x_0).
\ee 
This together with \eqref{eq.f.jyh} and the fact $J(\cdot,S_h)=J(\cdot,S)$ on $\X\setminus (l, x_0+h)$ implies that $S_h$ is a mild equilibrium and is strictly better than $S$.

{\bf Case (ii)} $x_0\in \partial S$ for boundary case (\ref{boundary.b}).
We denote
$$l:=\sup\{y<x_0, y\in S\},\quad r:=\inf\{y>x_0, y\in S\},\quad\text{and}\quad \tilde{\tau}_n:= \tau_{\left((l,r)\cap B(x_0, n)\right)}\wedge n\ \text{for }n\in\N.$$
By a similar discussion through \eqref{eq.lr.a}--\eqref{eq.bound.f1deriv'} (with Lemmas \ref{lm.peskir.ito} and \ref{lm.v.lrderiv} applied), we have that
\be\label{eq.V.peskir}
\begin{aligned}
&\E^{x_0}[V(\tilde{\tau}_n, X_{\tilde{\tau}_n},S)]-V(0,{x_0},S)\\
&\geq \E^{x_0}\left[\int_0^{\tilde{\tau}_n} \frac{1}{2}\left(\Lc V(s,X_{s}-,S)+\Lc V(s,X_{s}+,S)\right)ds\right]+a\E^{x_0}\left[\int_0^{\tilde{\tau}_n}\delta(t)dL^{x_0}_{t}\right].
\end{aligned}
\ee
By Lemma \ref{lm.v.c12}(a), 
$
\Lc V(t,x, S)=0$ for any $(t,x)\in [0,\infty)\times S^c,
$
and thus the first term on the RHS of \eqref{eq.V.peskir} vanishes for all $n\in \N$.
As a result, we can rewrite \eqref{eq.V.peskir} as
\bee\label{eq.prop.caseii} 
\E^{x_0}[V(\tilde{\tau}_n, X_{\tilde{\tau}_n},S)]-V(0,{x_0},S)\geq a\E^{x_0}\left[\int_0^{\tilde{\tau}_n}\delta(t)dL^{x_0}_{t}\right]\geq a\E^{x_0}\left[\int_0^{\tilde{\tau}_1}\delta(t)dL^{x_0}_{t}\right]>0,\quad \forall n\in\N.
\eee
Meanwhile, similar to \eqref{eq.prop.4.1'}, Assumption \ref{assume.f.positive} implies that $\E^{x_0}[V(\tilde{\tau}_n, X_{\tilde{\tau}_n},S)]\to \E^{x_0}[V(\tau_{(l,r)},X_{\tau_{(l,r)}},S )]$ as $n\to\infty$. This together with the above inequality implies that
\be\label{eq.prop.4.1}
\begin{aligned}
	&J(x_0,S\setminus\{x_0\})-f(x_0)=\E^{x_0}[V(\tau_{(l,r)},X_{\tau_{(l,r)}},S )]-V(0,{x_0},S)\geq a\E^{x_0}\left[\int_0^{\tilde{\tau}_1}\delta(t)dL^{x_0}_{t}\right]>0.
\end{aligned}  
\ee
Now set $\widetilde{S}:=S\setminus \{x_0\}$ and pick any $y\in (l,r)$. We can apply an argument similar to that in \eqref{eq.casei.step3.0}, by using \eqref{eq.prop.4.1} and replacing $S_h$ with $\widetilde{S}$, to reach that
$
J(y,\widetilde{S})-J(y,S)\geq 0
$. Hence, 
$
J(y,\widetilde{S})\geq f(y)$ for $y\in (l,r).
$
As $J(\cdot,\widetilde{S})=J(\cdot,S)$ on $\X\setminus(l,r)$, we have that $\widetilde{S}$ is a mild equilibrium. Due to \eqref{eq.prop.4.1}, $\widetilde{S}$ is strictly better than $S$.

{\bf Case (iii)} $x_0\in S^\circ$.
Choose $h_0>0$ such that $B(x_0, h_0)\setminus \{x_0\}\subset (\Gc\cap S^\circ)$. 
Following the argument in \textit{Step 1} of Case (i), we can again reach \eqref{eq.V.h} for some $0<h \leq h_0$, which indicates
$$
J(y,S\setminus B(x_0,h))>f(y),\quad \forall y\in B(x_0,h).
$$
As $J(\cdot,S\setminus B(x_0,h))=J(\cdot,S)$ on $\X\setminus B(x_0,h)$, we have that $S\setminus B(x_0,h)$ is a mild equilibrium and is strictly better than $S$.
\end{proof}

\begin{proof}[{\bf Proof of Theorem \ref{thm.optimalmild.weak}}]
Thanks to Lemma \ref{lm.v.c12}(a), Theorem \ref{prop.weak.cha} and Proposition \ref{prop.weak.optimalderiv}, we only need to show \eqref{eq.weak.equicha2} for $x\in S$. Recall $\Gc$ defined in \eqref{eq.assume.G}. Let $x_0\in S$ and we consider three cases: (i) $x_0\in (S^\circ\cap \Gc)$, (ii) $x_0=\theta_n\in S^\circ\setminus\Gc$ for some $n\in I$, and (iii) $x_0\in \partial S$.

{\bf Case (i)} $x_0\in (S^\circ\cap \Gc)$. We prove \eqref{eq.weak.equicha2} by contradiction. Suppose $\Lc V(0,x_0,S)=a>0$. By Assumption \ref{assume.f.positive}(ii), we can choose $h>0$ such that $V(t,x,S)=\delta(t)f(x)\in \Cc^{1,2}(B(x_0, h)\times(0,\infty))$ and
\be\label{eq.lv.sinner}   
		\Lc V(0,x, S)=\delta'(0)f(x)+\mu(x)f'(x)+\frac{1}{2}\sigma^2(x)f''(x)\geq \frac{a}{2},\quad\forall x\in B(x_0,h).
\ee
Then for any $(t,x)\in [0,\infty)\times B(x_0, h)$, we have that
\be\label{eq.lv.sinnsert}
	\begin{aligned}
		\Lc V(t,x,S)=&\delta'(t) f(x)+ \delta(t)\Big(\mu(x)f'(x)+\frac{1}{2}\sigma^2(x)f''(x)\Big)\\
		\geq &\delta(t)\Big(\delta'(0)f(x)+\mu(x)f'(x)+\frac{1}{2}\sigma^2(x)f''(x)\Big)
		\geq \delta(t)\frac{a}{2},
	\end{aligned}
\ee
where the first inequality above follows from Lemma \ref{lm.delta.inequ} and the non-negativity of $f$. 
Let us reuse the notation $\tau_n$ defined in \eqref{eq791}. By \eqref{eq.lv.sinnsert} and an argument similar to that for  \eqref{eq.bound.f1deriv'} and \eqref{eq.prop.4.1'} (notice that the local time integral in Lemma \ref{lm.peskir.ito} vanishes in the current case), we have that for any $x\in B(x_0,h)$,
\begin{align*}
	\begin{cases}
 \E^x[V(\tau_n, X_{\tau_n},S)]-V(0,x,S)
=\E^x\left[\int_0^{\tau_n} \Lc V(s,X_s,S)\right]
\geq \E^x\left[\int_0^{\tau_n} \delta(t)\frac{a}{2}dt\right]>0,\quad \forall n\in \N,\\
\E^x[V(\tau_{B(x_0,h)}, X_{\tau_{B(x_0,h)}},S)]=\lim_{n\to\infty} \E^x[V(\tau_n, X_{\tau_n},S)].
	\end{cases}
\end{align*}
This implies that
$$
\E^x[V(\tau_{B(x_0,h)}, X_{\tau_{B(x_0,h)}},S)-V(0,x,S)]\geq  \E^x\left[\int_0^{\tau_{B(x_0,h)}} \delta(t)\frac{a}{2}dt\right]>0,\quad \forall x\in B(x_0,h).
$$
Now consider $\widetilde{S}=S\setminus B(x_0,h)$. The above inequality implies
\be\label{eq.s'.hball}
\begin{aligned}
J(x, \widetilde{S})-f(x)=\E^x[V(\tau_{B(x_0,h)}, X_{\tau_{B(x_0,h)}},S)]-V(0,x,S)>0
\end{aligned}\quad \forall x\in B(x_0,h).
\ee
Obviously,
$J(\cdot, \widetilde{S})=J(\cdot,S)$ on $\X\setminus B(x_0,h).$
This together with \eqref{eq.s'.hball} shows that $\widetilde{S}$ is an equilibrium and is strictly better than $S$, a contradiction. Hence, $\Lc V(0,x_0,S)\leq 0$, as desired.

{\bf Case (ii)} $x_0=\theta_n\in S^\circ\setminus\Gc$ for some $n\in I$. Without loss of generality, we assume $\Lc V(0,x_0+,S)=a>0$. Then we can pick $h>0$ such that $(x_0, x_0+h)\subset S^\circ\cap (\theta_n, \theta_{n+1})$. By the continuity of $x\rightarrow \Lc V(0,x+,S)$ on $[x_0, x_0+h)$ (due to Assumptions \ref{assume.x.elliptic}(i), \ref{assume.f.positive}(ii) and the fact that $V (t, x, S)=\delta(t)f(x)$ for $x\in S$), we can find $0<\tilde{h}<h$ such that $\Lc V(0,y, S)\geq a/2>0$ for all $y\in (x_0, x_0+\tilde{h})$. Set $\tilde{x}:=(2x_0+\tilde{h})/2$. Then $B(\tilde{x}, \tilde{h}/4)\subset (x_0, x_0+\tilde{h})\subset S^\circ\cap \Gc$, and a contradiction can be reached by the same argument as in Case (i).

{\bf Case (iii)} $x_0\in \partial S$.  For boundary case (\ref{boundary.a}), suppose again that
$\Lc V(0,x_0-,S)\vee \Lc V(0,x_0+,S)>0.$
Without loss of generality, we assume $(x_0, x_0+h_0)\subset (S^\circ\cap \Gc)$ and $(x_0-h_0, x_0)\subset S^c$ for some $h_0>0$. By Lemma \ref{lm.v.c12}(a), $\Lc V(0,x-,S)\equiv 0$ on $(x_0-h_0, x_0]$, and therefore,
$
\Lc V(0,x_0+,S)>0.
$
Then the same argument as in Case (ii) can be applied to get a contradiction. 

For boundary case (\ref{boundary.b}), Lemma \ref{lm.v.c12}(a) directly tells that
$\Lc V(0,x-,S)\vee \Lc V(0,x+,S)=0,$
and the proof is complete.
\end{proof}

\section{When Weak or Optimal Mild Equilibria are Strong}\label{sec:strong.optimal}
After establishing the relation between optimal mild and weak equilibria, we take a further step to study whether a weak or optimal mild equilibrium is strong. 

	We already know that an admissible weak or optimal mild equilibrium $S$ satisfies the two conditions \eqref{eq.weak.equicha1} and \eqref{eq.weak.equicha2} in Theorem \ref{prop.weak.cha}. To make $S$ a strong equilibrium, the first order condition \eqref{e7} needs to be upgraded to the local maximum condition \eqref{e8}. Recall the discussion at the beginning of Section \ref{subsec:3.1}. Intuitively, a sufficient condition for \eqref{e8} is the LHS of \eqref{eq.explain.31.2} being negative for all $\eps$ small enough. As a result, if at least one of the two inequalities \eqref{eq.weak.equicha1} and \eqref{eq.weak.equicha2} is strict for all the points in the weak or optimal equilibrium $S$, then $S$ should also be strong. To this end, let us define for any admissible $S\in \Bc$, 
	\be\label{eq.define.Sk}
	\begin{aligned}
		\mathfrak{S}_S:=&\{x\in S: \Lc V(0,x-,S \wedge \Lc V(0,x+,S)<0\}\cup\{x\in S: V_x(0,x-,S)>V_x(0,x+,S)\}.
	\end{aligned}
	\ee
Theorem \ref{thm.weak.strong} and Theorem \ref{thm.strong.optimal} are the main results of this section, and their proofs are provided in the next subsection. The first main result concerns when a weak equilibrium is strong.

\begin{Theorem}\label{thm.weak.strong}
Let Assumptions \ref{assume.x.elliptic}--\ref{assume.new} hold and $S$ be an admissible weak equilibrium. If $S=\mathfrak{S}_S$, then $S$ is also strong.
\end{Theorem}

The next result regards the relation between optimal mild and strong equilibria.

\begin{Theorem}\label{thm.strong.optimal}
Let Assumptions \ref{assume.x.elliptic}--\ref{assume.new} hold. 
\bi 
\item[(a)] For any admissible optimal mild equilibrium $S$,  if $\Sk_S$ is admissible and closed, then $\Sk_S$ is a strong equilibrium.
\item[(b)] Recall $S^*$ defined in \eqref{eq.nota.sstar}. We have $\overline{\Sk_{S^*}}=S^*$. Hence, if $\Sk_{S^*}$ is closed and admissible, then $S^*$ is a strong equilibrium.
 \ei
\end{Theorem}

\begin{Remark}\label{rm.sk.best}
Theorem \ref{thm.strong.optimal} indicates that $S^*$ and $\Sk_{S^*}$ are almost the same, and roughly speaking, $S^*$ is a strong equilibrium possibly except some points in $\overline{\Sk_{S^*}}\setminus \Sk_{S^*}$. In many cases we indeed have $S^*=\Sk_{S^*}$, as a result of which $S^*$ is strong. This is demonstrated in all the examples in Section~\ref{sec:eg}.
\end{Remark}

\begin{Remark}
Suppose Assumptions \ref{assume.x.elliptic}--\ref{assume.new} hold and $S^*$ is admissible. Then $S^*$ cannot contain an isolated point at which $f$ is continuously differentiable. Indeed, suppose $x$ is an isolated point of $S^*=\overline\Sk_{S^*}$ and $f$ is smooth at $x$. Then $x\in\Sk_{S^*}$. On the other hand, since $\mathcal{L}(0,x-,S^*)=\mathcal{L}(0,x+,S^*)=0$ by Lemma \ref{lm.v.c12}(a), and $V_x(0,x-,S^*)=V_x(0,x+,S)$ by Corollary \ref{cor.boundary.smoothfit}, we would have $x\notin\Sk_{S^*}$, a contradiction. 
\end{Remark}

\subsection{Proofs of Theorems \ref{thm.weak.strong} and \ref{thm.strong.optimal}}

As discussed above, we aim to achieve the negativity in the RHS of \eqref{eq.explain.31.2} for $\eps$ small enough; when $V_x(s,x+,S)-V_x(s,x-,S)=0$, the integral on $\frac{1}{2}(\Lc V(s,X_{s}-,S)+\Lc V(s, X_s+,S))$ on the RHS of \eqref{eq.explain.31.2} should be negative. Since only one of the two values $\Lc V(s,X_{s}\pm,S)$ is required to be negative in the definition of $\Sk_S$, we will estimate the probability that $X$ goes to the left/right from the starting point. Such probability estimation is provided in the following lemma.

\begin{Lemma}\label{lm.1/2}
Let Assumption \ref{assume.x.elliptic} hold. Then
\be\label{eq.lm.5.1} 
\lim_{t\searrow 0} \P^{x_0}(X_t>x_0)=\lim_{t\searrow 0} \P^{x_0}(X_t<x_0)=\frac{1}{2},\quad \forall x_0\in \X.
\ee 
\end{Lemma}

\begin{proof}
Let $X_0=x_0\in\X$. Recall $\widetilde{X}$ and $\bar X$ defined in \eqref{e11}. Denote $R_\eps:=\mu(x_0)\eps+\bar X_\eps$. Then 
\be\label{eq.lm5.1} 
X_\eps = x_0+R_\eps+\sigma(x_0)W_\eps.
\ee  
By Lemma \ref{lm.step.2}, there exists some constant $C>0$ such that for any $\eps>0$ small enough, 
$
\E^{x_0}[|R_\varepsilon|]\leq C \varepsilon,
$
which leads to
\be\label{eq.lm5.1'}
\P^{x_0}\left(|R_\varepsilon|\geq\frac{1}{2}\varepsilon^{3/4}\right)\leq \frac{2\E^{x_0}[|R_\varepsilon|]}{\varepsilon^{3/4}}\leq 2C\cdot \varepsilon^{1/4}.
\ee 
By \eqref{eq.lm5.1} and \eqref{eq.lm5.1'}, for $\varepsilon>0$ small enough,
\be\notag
\begin{aligned}
&\P^{x_0}(X_\varepsilon>x_0)\geq \P^{x_0} \left(\sigma(x_0) W_{\varepsilon} >\varepsilon^{3/4}, R_\varepsilon >-\frac{1}{2} \varepsilon^{3/4}\right)\\
&\geq \P^{x_0} \left(\sigma(x_0) W_{\varepsilon} >\varepsilon^{3/4}\right)- \P^{x_0}\left(R_\varepsilon \leq-\frac{1}{2} \varepsilon^{3/4}\right)
\geq 1-\Phi\left(\frac{\varepsilon^{3/4}}{\sigma(x_0)\sqrt{\varepsilon}}\right)-\P^{x_0}\left(|R_\varepsilon|\geq\frac{1}{2}\varepsilon^{3/4}\right)\\
& \geq  1-\Phi\left(\frac{\varepsilon^{1/4}}{\sigma(x_0)}\right)-2C\eps^{1/4}\rightarrow 1-\Phi(0)-0=\frac{1}{2}, \quad \text{as}\; \eps \searrow 0,
\end{aligned}
\ee 
where $\Phi$ is the cumulative distribution function for the standard normal distribution. Therefore,
$
\liminf_{t\searrow 0} \P^{x_0}(X_t>x_0)\geq\frac{1}{2}.
$
Similarly,
$
\liminf_{t\searrow 0} \P^{x_0}(X_t<x_0)\geq\frac{1}{2}.
$
Thus, \eqref{eq.lm.5.1} holds.
\end{proof}

Now we are ready to prove Theorem \ref{thm.weak.strong}.
\begin{proof}[{\bf Proof of Theorem \ref{thm.weak.strong}}]
To prove the desired result, we need to verify that for any $x_0\in S$,
\be\label{eq.strong.x0} 
\exists	\varepsilon(x_0)>0, \text{ s.t. } \forall \varepsilon'\leq \varepsilon(x_0), f(x_0)-\E^{x_0}[\delta(\rho^{\varepsilon'}_S) f(X_{\rho^{\varepsilon'}_S})] \geq 0.
\ee
Since $S$ is a weak equilibrium, by Theorem \ref{prop.weak.cha},
$$V_x(0,x_0-,S)-V_x(0,x_0+,S)\geq0,\quad\forall x_0\in S.$$
Recall \eqref{eq.define.Sk} and $\Gc$ defined in \eqref{eq.assume.G}. Pick $x_0\in \Sk_S$, and we shall verify \eqref{eq.strong.x0} for two cases: (i) $V_x(0,x_0-,S)-V_x(0,x_0+,S)>0$, and (ii) $V_x(0,x_0-,S)-V_x(0,x_0+,S)=0$.

{\bf Case (i)} Suppose $a:=V_x(0,x_0-,S)-V_x(0,x_0+,S)>0$. By the continuity of $t\mapsto V_x(t,x_0\pm,S)$, we take $\eps>0$ small enough such that $\delta(t)>\frac{1}{2}$ for all $t\in(0,\eps)$, and 
\be\label{eq.thm.i1} 
V_x(t,x_0+,S)-V_x(t,x_0-,S)<-\frac{a}{2},\quad \forall\,t\in(0,\eps).
\ee
Let $h>0$ such that both $(x_0-h,x_0)$ and $(x_0,x_0+h)$ belong to $\Gc$. Then for $\varepsilon$ small enough,
\be\label{eq.isolate.sk}
\begin{aligned}
	&\E^{x_0}[V(\varepsilon, X_\varepsilon, S)]-V(0,x_0, S)+o(\varepsilon)=\E^{x_0}[V(\varepsilon\wedge \tau_{B(x_0,h)},X_{ \varepsilon\wedge \tau_{B(x_0,h)}},S)]-V(0,x_0, S)\\
	&\leq\E^{x_0}\left[\int_0^{\eps\wedge\tau_{B(x_0,h)}} \frac{1}{2}(\Lc V(s,X_s-,S)+\Lc V(s,X_s+, S))ds\right]-\frac{a}{4}\E^{x_0}[L^{x_0}_{\tau_{B(x_0,h)}\wedge \varepsilon}],
\end{aligned}
\ee
where the first (in)equality follows from Lemma \ref{lm.varepsilon.h}, the second (in)equality follows from Lemma \ref{lm.peskir.ito} and \eqref{eq.thm.i1} (the diffusion term vanishes after taking expectation due to the boundedness of $V_x\sigma$ on $[0,\eps]\times \overline{B(x_0,h)}$).
By Lemma \ref{lm.v.c12}(b), there exists a constant $K>0$ such that
$$\sup_{(t,y)\in [0,1]\times \overline{B(x_0,h)}}\frac{1}{2}|\Lc V(t,y+, S)+\Lc V(t,y-, S)|\leq K,$$ 
Then by Lemma \ref{lm.local.time} and $|\sigma(x_0)|>0$, we can take $\eps_0\in(0,1)$ such that for any $\varepsilon\in(0,\eps_0)$, $\frac{a}{4\varepsilon}\E^{x_0}[L^{x_0}_{\tau_{B(x_0,h)}\wedge \varepsilon}]\geq (K+1)$ and the term $o(\varepsilon)$ in \eqref{eq.isolate.sk} satisfies $|o(\varepsilon)|\leq \frac{1}{2}\varepsilon$. Hence, \eqref{eq.isolate.sk} leads to
$$
\begin{aligned}
&\E^{x_0}[\delta(\rho^{\varepsilon}_S) f(X_{\rho^{\varepsilon}_S})]-f(x_0) =\E^{x_0}[V(\varepsilon, X_\varepsilon, S)]-V(0,x_0, S)\\
&\leq K\eps-\frac{a}{4}\E^{x_0}[L^{x_0}_{\tau_{B(x_0,h)}\wedge \varepsilon}]+\frac{1}{2}\eps\leq -\frac{1}{2}\varepsilon,\quad \forall \varepsilon\leq \varepsilon_0.
\end{aligned}
$$

{\bf Case (ii)} Suppose $V_x(0,x_0-,S)-V_x(0,x_0+,S)=0$. Then by \eqref{eq.vx.lineart},
$$|V_x(t,x_0+,S)-V_x(t,x_0-,S)|=o(\sqrt{t}) \quad \text{for $t>0$ small enough}.$$
This together with Lemma \ref{lm.local.time} leads to
\begin{equation}\label{eq.localtime.o} 
\E^{x_0}\left|\frac{1}{2}\int_0^{\varepsilon\wedge \tau_{B(x_0,h)}}(V_x(s,x_0+,S)-V_x(s,x_0-,S))dL^{x_0}_s\right|=o(\sqrt{\eps})\cdot\E^{x_0}\left[L_{\eps\wedge\tau_{B(x_0,h)}}^{x_0}\right]=o(\eps).
\end{equation}
Choose $h_0>0$ such that $(x_0-h_0,x_0)\cup(x_0,x_0+h_0)$ is contained in $(S^\circ\cap\Gc)\cup ( \X\setminus S)$. For any $h\in(0,h_0)$, similar to \eqref{eq.isolate.sk}, we apply Lemmas \ref{lm.peskir.ito}, \ref{lm.varepsilon.h} and then combine with \eqref{eq.localtime.o} to get
\be\label{eq.vh.lv}  
\begin{aligned}
	\E^{x_0}[V(\varepsilon, X_\varepsilon, S)]-V(0,x_0,S)
	=&\E^{x_0}[V(\tau_{B(x_0,h)}\wedge \varepsilon, X_{\tau_{B(x_0,h)}\wedge \varepsilon}, S)]-V(0, x_0,S)+o(\varepsilon)\\
	=&\E^{x_0}\left[\int_0^{\tau_{B(x_0,h)}\wedge \varepsilon} \Lc V(s,X_s, S)ds\right]+o(\varepsilon).
\end{aligned}
\ee
Since $V_x(0,x_0-,S)-V_x(0,x_0+,S)=0$ and $x_0\in \Sk_S$, we have
\be\label{eq.min.S} 
\Lc V(0,x_0-,S) \wedge  \Lc V(0,x_0+,S)<0.
\ee
Without loss of generality, we can assume that
\bee
-A:=\Lc V(0,x_0+,S)<0\quad \text{and}\quad \Lc V(0,x_0-,S)\leq 0.
\eee
By the (left/right) continuity of $(t,x)\mapsto \mathcal{L}V(t,x\pm,S)$ at $(0,x_0)$, we can choose $h\in(0,h_0)$ and $\eps_0>0$ small enough, such that for any $t\in[0,\eps_0]$, $x\in(x_0,x_0+h)$ and $y\in(x_0-h,x_0)$,
\be\label{eq.thm.ii1} 
\Lc V(t,x,S)=\Lc V(t,x+,S)\leq-\frac{A}{2}\quad\text{and}\quad \Lc V(t,y,S)=\Lc V(t,y-,S)\leq\frac{A}{8}.
\ee
Then for $\eps\in(0,\eps_0)$ small enough, the first inequality in \eqref{eq.thm.ii1} implies that
 \be\label{eq.lv.right0} 
\begin{aligned}
&\E^{x_0}\left[\int_0^{\tau_{B(x_0,h)}\wedge \varepsilon} \Lc V(t,X_t,S)1_{\{X_t>x_0\}}dt\right]
\leq -\frac{A}{2}\E^{x_0}\left[\int_0^{\tau_{B(x_0,h)}\wedge \varepsilon} 1_{\{X_t>x_0\}}dt\right]\\
&= -\frac{A}{2}\E^{x_0}\left[\int_0^{\varepsilon} 1_{\{X_t>x_0\}}dt\right]+\frac{A}{2}\E^{x_0}\left[\int_{\tau_{B(x_0,h)}\wedge \varepsilon}^\eps 1_{\{X_t>x_0\}}dt\right]\\
&= -\frac{A}{2}\int_0^{\varepsilon} \P^{x_0}(X_t>x_0)dt+\frac{A}{2}\E^{x_0}\left[(\eps-\tau_{B(x_0,h)})1_{\{\tau_{B(x_0,h)}<\eps\}}\right]\\
&\leq -\frac{A}{5}\eps+\frac{A}{2}\E^{x_0}\left[(\eps-\tau_{B(x_0,h)})1_{\{\tau_{B(x_0,h)}<\eps\}}\right]\\
&\leq -\frac{A}{5}\eps+\frac{A}{2}\eps\P^{x_0}\left(\tau_{B(x_0,h)}<\eps\right)= -\frac{A}{5}\eps+\frac{A}{2}\eps \cdot o(\eps)\leq -\frac{A}{6}\eps,\\
\end{aligned}
\ee
where the forth (in)equality above follows from Lemma \ref{lm.1/2}, and the sixth (in)equality follows from Lemma \ref{lm.step.1}.
In addition, the second inequality in \eqref{eq.thm.ii1} implies 
\be\label{eq.lv.left} 
\begin{aligned}
	\E^{x_0}\left[\int_0^{\tau_{B(x_0,h)}\wedge \varepsilon} \Lc V(t,X_t,S)1_{\{X_t<x_0\}}dt\right]
	\leq \frac{A}{8}\varepsilon.
\end{aligned}
\ee
Therefore, by plugging \eqref{eq.lv.right0} and \eqref{eq.lv.left} into \eqref{eq.vh.lv}, we have that for $\eps>0$ small enough,
\be\notag
	\E^{x_0}[\delta(\rho^{\varepsilon}_S) f(X_{\rho^{\varepsilon}_S})]-f(x_0)=\E^{x_0}[V(\varepsilon, X_\varepsilon, S)]-V(0,x_0,S)\leq -\frac{A}{6}\eps+\frac{A}{8}\varepsilon+o(\varepsilon)<-\frac{A\eps}{25},
\ee
and the proof is complete.
\end{proof}

To prepare for the proof of Theorem \ref{thm.strong.optimal}, let us illustrate a property of an arbitrary optimal mild equilibrium $S$, which says that $\Sk_S$ actually forms the ``essential" part of $S$, and by removing the ``inessential" part from $S$ the remaining part is still optimal mild. 

\begin{Proposition}\label{prop.inner.S}
	Let Assumptions \ref{assume.x.elliptic}--\ref{assume.new} hold. For any admissible optimal mild equilibrium $S$,  $\overline{\Sk_S}$ is also optimal mild. In addition, if $\overline{\Sk_S}$ is admissible, then $\overline{\Sk_S}$ is a weak equilibrium.
\end{Proposition}

\begin{proof}
\textit{Step 1.} We first characterize $S\setminus \overline{\Sk_S}$. 
As $S$ is admissible, we can write $S$ as a union of disjoint closed intervals
\be\label{eq.S.structure} 
S=\cup_{n\in \Lambda_1} [\alpha_{2n-1}, \alpha_{2n}], \;\text{where}\; \alpha_{2n-1}\leq \alpha_{2n}< \alpha_{2n+1}.\footnote{These intervals are understood as the ones restricted in $\X$ in a natural way, e.g., one $[\alpha_{2n-1},\alpha_{2n}]$ could be $[\alpha_{2n-1},\infty)$ if $\sup\X=\infty$.}
\ee
where $\Lambda_1\subset\Z$ is either a finite or countable subset.  Since $S$ is closed, we have that $\overline{\Sk_S}\subset S$.  For each $n\in \Lambda_1$, by the closeness of $\overline{\Sk_S}$, we can see that $[\alpha_{2n-1}, \alpha_{2n}]\setminus \overline{\Sk_S}$ consists of at most countably many disjoint intervals $(I_{n_k})_{k}$ of the following four forms:
\be\label{eq.four.forms} 
1.\; [\alpha_{2n-1}, \gamma); \quad  2.\; (\gamma', \alpha_{2n}];  \quad  3.\; (\beta,\beta');  \quad 4.\; [\alpha_{2n-1},\alpha_{2n}].
\ee 	
For each $I_{n_k}$ of the four forms in \eqref{eq.four.forms}, we define an open interval $(l_{n_k},r_{n_k})$ as follows
\be\label{eq.four.lr}  
\begin{cases}
	\begin{array}{ll}
		1.\; l_{n_k}:=\sup\{y<\alpha_{2n-1}, y\in \Sk_S\},& r_{n_k}:=\gamma;\\
		2.\; l_{n_k}:=\gamma', &r_{n_k}:=\inf\{y>\alpha_{2n}, y\in \Sk_S\};\\
		3.\; l_{n_k}=\beta, & r_{n_k}:=\beta';\\
		4.\; l_{n_k}:=\sup\{y<\alpha_{2n-1}, y\in \Sk_S\}, &r_{n_k}:=\inf\{y>\alpha_{2n}, y\in \Sk_S\},
	\end{array}
\end{cases}
\ee
and set $\sup\emptyset:=\inf\X$ and $\inf\emptyset:=\sup\X$ if it happens. Notice that each two of those open intervals $((l_{n_k}, r_{n_k}))_{n,k}$ are either disjoint or identical, and $l_{n_k}$ can be $-\infty$ (resp. $r_{n_k}$ can be $\infty$). Since the total number of these intervals $((l_{n_k}, r_{n_k}))_{n,k}$ is at most countable, we omit the repeating ones and re-index them as $((l_k, r_k))_{k\in \Lambda}$ such that they are disjoint and $\Lambda\subset \Z$ is either a finite or countable subset. Then
$S\setminus (\cup_{k\in \Lambda} (l_k, r_k))=\overline{\Sk_S}.$

\textit{Step 2.} We prove that for each $k\in \Lambda$, 
\be\label{eq.v-Ik,optimal}
J(x,S\setminus (l_k, r_k))=J(x,S),\quad \forall x\in (l_k, r_k).
\ee
Fix $k\in \Lambda$. \textit{Step 1} tells that for any $x\in (l_k, r_k)$, $x$ either belongs to $S\setminus \overline{\Sk_S}$ or belongs to $S^c$. 

	(1) If $x\in S^c$ or $x\in \partial S$ for boundary case (\ref{boundary.b}), Lemma \ref{lm.v.c12} tells that
\be\label{eq.I'.lv}   
\Lc V(t,x+,S)\equiv\Lc V(t,x-,S)\equiv0\quad \forall t\in[0,\infty).
\ee 

(2) Suppose $x\in S^\circ\setminus \overline{\Sk_S}$. By the fact that $S$ is an admissible optimal mild equilibrium, Theorem \ref{thm.optimalmild.weak} tells that $S$ is also weak. Then \eqref{eq.weak.equicha2} together with the definition of $\Sk_S$ leads to 
$$\Lc V(0,x-,S)=\Lc V(0,x+,S)=0;\quad  V(t,x,S)=\delta(t)f(x)\quad \forall t\geq 0.$$
Then by a similar argument as in \eqref{eq.lv.sinnsert} (with $\frac{a}{2}$ replaced by 0), we reach that
\be\label{eq.I.lv}  
\Lc V(t,x-,S)\wedge \Lc V(t,x+,S)\geq 0 \quad \forall t\in [0,\infty).
\ee 

(3) Otherwise, $x\in \partial S\setminus \overline{\Sk_S}$ of boundary case (\ref{boundary.a}), and for this case, we can also deduce \eqref{eq.I.lv} by a combination of cases (1) \& (2). 

In sum, we have
\be\label{eq.digx.lv} 
\frac{1}{2}(\Lc V(t,x-,S)+\Lc V(t,x+,S))\geq 0\quad \forall (t,x)\in [0,\infty)\times (l_k,r_k).
\ee


Recall $(\theta_i)_{i\in I}$ defined in Assumption \ref{assume.f.positive}(ii). By Proposition \ref{prop.weak.optimalderiv} and the definition of $\Sk_S$, $V_x(0,\theta_i+, S)= V_x(0,\theta_i-,S)$ for each $\theta_i\in ((l_k,r_k)\cap S)$. Then for any $n\in\N$ and $\theta_i\in (l_k,r_k)\cap B(x_0,n)$, no matter $\theta_i$ belongs to $S\setminus \overline{\Sk_S}$ or $S^c$, from the fact that $V(x,t,S)=\delta(t)f(x)$ for $x\in S$ and Lemma \ref{lm.v.lrderiv}, we have that
\be\label{eq.digx.lrderiv}  
V_x(t,\theta_i+, S)-V_x(t,\theta_i-,S)\geq \delta(t)( V_x(0,\theta_i+, S)-V_x(0,\theta_i-,S))=0,\quad \forall t\geq 0.
\ee

Note that for each $n\in\N$ the interval $B(x_0, n)\cap (l_k, r_k)$ contains at most finite points $\theta_i$. Now take $x_0\in (l_k, r_k)$ and denote
$\tau_n:=\tau_{(l_k,r_k)\cap B(x_0,n)}\wedge n$ for $n\in\N$. 
By Lemma \ref{lm.peskir.ito},
\be\notag
\begin{aligned}
	&V(\tau_n, X_{\tau_n},S)-V(0,x_0,S)= \int_0^{\tau_n} \frac{1}{2}(\Lc V(t,X_t-,S)+\Lc V(t,X_t+,S))dt\\
	&\quad +\int_0^{\tau_n} V_x(t, X_t,S)\sigma(X_s)\cdot 1_{\{X_t\neq \theta_i,\forall\, i\}}dW_t+\frac{1}{2}\sum_{\theta_i\in (l_k, r_k)}\int_0^{\tau_n}(V_x(t,\theta_i+,S)-V_x(t,\theta_i-,S))dL^{\theta_i}_{t},
\end{aligned}
\ee
Taking expectation for the above and combining with \eqref{eq.digx.lv} and \eqref{eq.digx.lrderiv}, we have that
$$
	\E^{x_0}[V(\tau_n, X_{\tau_n}, S)]-V(0,x_0,S)\geq 0.
$$
Similar to \eqref{eq.prop.4.1'}, we can show that $\lim_{n\to\infty}\E^{x_0}[V(\tau_n, X_{\tau_n}, S)]=\E^{x_0}[V(\tau_{(l_k, r_k)}, X_{\tau_{(l_k, r_k)}}, S)]$. This together with the above inequality implies that
\begin{align*}
J(x_0, S\setminus (l_k, r_k))-J(x_0, S)= \E^{x_0}[V(\tau_{(l_k, r_k)}, X_{\tau_{(l_k, r_k)}}, S)]- V(0,x_0,S)\geq 0.
\end{align*}
By the arbitrariness of $x_0\in (l_k, r_k)$, we have
$
J(x, S\setminus (l_k, r_k))\geq J(x, S)$ for all $x\in (l_k, r_k).
$
Meanwhile, $J(x, S\setminus (l_k, r_k))=J(x, S)$ for $x\in \X\setminus (l_k, r_k)$. Then by the optimality of $S$, $S\setminus (l_k,r_k)$ is also an optimal mild equilibrium, and thus \eqref{eq.v-Ik,optimal} follows.

\textit{Step 3.} We show $\overline{\Sk_S}$ is optimal mild. By \textit{Step 2}, \eqref{eq.v-Ik,optimal} holds for all $k\in \Lambda$. From the construction of  the intervals $(l_k, r_k)_{k\in \Lambda}$ in \eqref{eq.four.lr}, we can see that removing one of them does not change the values of function $J$ on the rest parts, that is, for any $k\in \Lambda$,
	$$J(x,\overline{\Sk_S})=J(x,S\setminus (l_k, r_k))\quad \forall x\in (l_k,r_k).$$
Hence, we can conclude that for any $k\in \Lambda$,
$$
J(x,\overline{\Sk_S})=J(x,S\setminus (l_k,r_k))=J(x,S), \quad \forall x\in (l_k, r_k).
$$
As $J(x, \overline{\Sk_S})=f(x)=J(x,S)$ for all $x\in\overline{\Sk_S}$,
$$
J(x,\overline{\Sk_S})=J(x,S),\quad \forall x\in \X.
$$
This implies $\overline{\Sk_S}$ is an optimal mild equilibrium. By Theorem \ref{thm.optimalmild.weak}, if $\overline{\Sk_S}$ is admissible then it is also a weak equilibrium.
\end{proof}

Thanks to Theorem \ref{thm.weak.strong} and Proposition \ref{prop.inner.S}, we are ready to prove Theorem \ref{thm.strong.optimal}.

\begin{proof}[{\bf Proof of Theorem \ref{thm.strong.optimal}}]
	
\textbf{Part (a):} Suppose $S$ is an optimal mild equilibrium and $\Sk_S$ is closed and admissible. Proposition \ref{prop.inner.S} tells that $\Sk_S=\overline{\Sk_S}$ is both an optimal mild and weak equilibrium.  Then by Theorem \ref{thm.weak.strong}, to prove that $\Sk_S$ is strong, it is sufficient to verify that 
$
\Sk_S=\Sk_{\Sk_S}.
$
Notice that $\Sk_{\Sk_S}\subset \Sk_S$. Take $x_0\in \Sk_S$ and we show $x_0\in \Sk_{\Sk_S}$. If $V_x(0,x_0,\Sk_S)>V_x(0,x_0,\Sk_S)$, then $x_0\in \Sk_{\Sk_S}$. Otherwise, $V_x(0,x_0,\Sk_S)=V_x(0,x_0,\Sk_S)$, and it remains to verify that
\be\label{eq.thm.stronga1} 
\Lc V(0,x_0-,\Sk_S) \wedge \Lc V(0,x_0+,\Sk_S)<0.
\ee 
Since both $S$ and $\Sk_S$ are optimal mild, we have
$$V(0,x,\Sk_S)\equiv J(x,\Sk_S)\equiv J(x,\overline{\Sk_S})\equiv J(x,S)\equiv V(0,x,S)\quad\forall x\in\X.$$ 
Then 
\be\label{eq,thm.stronga} 
V_x(0,x_0-,S)-V_x(0,x_0+,S)=V_x(0,x_0-,\Sk_S)-V_x(0,x_0+,\Sk_S)=0.
\ee
Since $x_0\in \Sk_S$, by the definition of $\Sk_S$, \eqref{eq,thm.stronga} leads to that
\be\label{eq.thm.strong}  
\Lc V(0,x_0-,S)\wedge \Lc V(0,x_0+,S)<0.
\ee 
This together with \eqref{eq.lm.lv0} implies that $x_0$ cannot be an isolated point of $S$. We consider the following two cases.

(1) Suppose $x_0\in S^\circ$. Note that $V(t,x,S)=\delta(t) f(x)$ on $S$.  Then by \eqref{eq.thm.strong}, without loss of generality we assume $\Lc V(0,x_0+,S)=\Lc (\delta f) (0,x_0+)<0$. By the right continuity of $x\mapsto \Lc (\delta f) (0,x_+)$ at $x_0$, we can find $h>0$ small enough such that $[x_0,x_0+h)\subset (S^\circ \cap \Gc)$ (recall $\Gc$ defined in \eqref{eq.assume.G}) and
\be\label{eq.thm.stronga}  
\Lc (\delta f) (0,x_+)<0,\quad \forall x\in[x_0,x_0+h).
\ee
Hence, $[x_0,x_0+h)\subset \Sk_S$, and thus $\Lc V(0,x_0+,\Sk_S)=\Lc (\delta f) (0,x_+)<0$. 

(2) Otherwise, $x_0\in \partial (S^\circ)$ for boundary case (\ref{boundary.a}). Without loss of generality, we assume $(x_0,x_0+h)\subset (S^\circ \cap \Gc)$ for $h>0$ small enough. Then by \eqref{eq.thm.strong} and \eqref{eq.lm.lv0}, we again have $\Lc V(0,x_0+,S)=\Lc (\delta f) (0,x_0+)<0$. A similar discussion as in case (1) implies $\Lc V(0,x_0+,\Sk_S)=\Lc (\delta f) (0,x_0+)<0$.

In sum, \eqref{eq.thm.stronga1} holds, and the proof of part (a) is complete.

\textbf{Part (b):} Lemma \ref{lm.optimal.mild} indicates that $S^*$ is an optimal mild equilibrium. Then by Proposition \ref{prop.inner.S}, $\overline{\Sk_{S^*}}\subset S^*$ is an optimal mild equilibrium. As $S^*$ is the smallest optimal mild equilibrium, $\overline{\Sk_{S^*}}=S^*$. The rest statement directly follows from part (a).
\end{proof}

\section{Examples}\label{sec:eg}
In this section, we provide three examples to demonstrate our results. In the first example, we have two strong equilibria, one of which is not optimal mild. This indicates that an strong equilibrium may not be optimal mild. In the second example, we show a weak equilibrium may not be strong. The third example is the stopping for an American put option on a geometric Brownian motion, in which we provide all three types of equilibria.

\subsection{An example showing  optimal mild $\subsetneqq$ strong}\label{sebsec:eg.counter}
In this subsection, we construct an example where the set of optimal mild equilibria is strictly contained (i.e., $\subsetneqq$) in the set of strong equilibria. Let $dX_t=dW_t$ and thus $X$ is a Brownian motion with $\X=\R$. Take discount function $\delta(t)=\frac{1}{1+\beta t}$. Let $a<b$, $0<c<d$ such that
\be\label{eq.counter.abcd}
\dfrac{\int_0^\infty e^{-s}\frac{\sqrt{2\beta s}}{\sinh((b-a)\sqrt{2\beta s})} ds}{\sqrt{\frac{\pi \beta }{2}}+\int_0^\infty e^{-s}\sqrt{2\beta s}\coth((b-a)\sqrt{2\beta s}) ds}<\dfrac{c}{d}<\int_0^\infty e^{-(s+(b-a)\sqrt{2\beta s})}ds.
\ee 
Notice that such parameters do exist, e.g., let $b-a=1$, then for \eqref{eq.counter.abcd}, we have LHS $\approx0.3952<\frac{c}{d}<0.4544\approx$ RHS. 

Define
\begin{align}
	\notag J_b(x)&:=d\E^x[\delta(\rho_{\{b\}})]=d\int_0^\infty \frac{p(t)}{1+\beta t} dt=d\int_0^\infty \int_0^\infty e^{-(1+\beta t)s}p(t)dsdt\\
	\label{e35} &=d\int_0^\infty e^{-s}\E^x[e^{-\beta s \rho_{\{b\}}}]ds=  d\int_0^\infty e^{-s}e^{-|x-b|\sqrt{2\beta s}}ds,\quad x\in\X.
\end{align}
where the second line uses the formula in \cite[2.0.1 on page 204]{MR1912205}. We further define
\begin{align}
	\notag	J_{ab}(x)&:=c\E^x[\delta(\rho_{\{a,b\}})\cdot 1_{\{\rho_{\{a,b\}=a}\}}]+d\E^x[\delta(\rho_{\{a,b\}})\cdot 1_{\{\rho_{\{a,b\}=b}\}}]\\
	&=\label{e35'}	\begin{cases}	
		c\int_0^\infty e^{-s}e^{-|x-a|\sqrt{2\beta s}}ds,&x<a,\\
		c\int_0^\infty e^{-s}\frac{\sinh((b-x)\sqrt{2\beta s})}{\sinh((b-a)\sqrt{2\beta s})}ds+  d\int_0^\infty e^{-s}\frac{\sinh((x-a)\sqrt{2\beta s})}{\sinh((b-a)\sqrt{2\beta s})}ds,&a\leq x\leq b,\\
		J_b(x),&x>b.
	\end{cases}
\end{align}
where the expression for $J_{ab}$ on $[a,b]$ is obtained by the formula in \cite[3.0.5 (a)\&(b) on page 218]{MR1912205} combined with an argument similar to that in \eqref{e35}. 
Let $f$ be any function satisfying Assumption \ref{assume.f.positive} such that
\be\label{eq.eg1.f} 
f(a)=c, f(b)=d; \quad  f(x)< \min\{J_b(x), J_{ab}(x)\}, \quad \forall\, x\in \X\setminus \{a,b\}.
\ee 
Note that
\be\label{eq.cd} 
c<d\int_0^\infty e^{-(s+(b-a)\sqrt{2\beta s})}ds=J_b(a),
\ee
which shows such function $f(x)$ indeed exists. \footnote{By the strong markov property of $X$ and \eqref{eq.cd}, one can easily check that $J_{ab}(x)<J_b(x)$ for $x<b$. Hence a quick example for such $f$ would be: 
	$$
	f(x):=
	\begin{aligned}
	\begin{cases}
		\frac{1}{1+(a-x)}J_{ab}(x), &x\leq a,\\
\frac{1}{1+(x-a)(b-x)}J_{ab}(x), &a<x\leq b,\\
\frac{1}{1+(x-b)}J_b(x), &x>b.
\end{cases}
	\end{aligned}
$$}

One can easily verify that Assumptions \ref{assume.x.elliptic}--\ref{assume.f.positive} hold. Moreover, Assumption \ref{assume.new} is also satisfied due to Lemma \ref{rm.suff.newassume} and Remark \ref{rm.weight.discount}. We have the following result. 
\begin{Proposition}
	$ \{b\}$ is the unique optimal mild equilibrium, while both $\{b\}$ and $\{a, b\}$ are strong equilibria.
\end{Proposition}

\begin{proof}
	Recall $S^*$ defined in \eqref{eq.nota.sstar}. First notice that
	$$J_b(x)=J(x,\{b\})\quad\text{and}\quad J_{ab}(x)=J(x,\{a,b\}),\quad\forall\,x\in\X.$$
	Then by \eqref{eq.eg1.f} and \eqref{eq.cd}, it is easy to see that both $\{a,b\}$ and $\{b\}$ are mild equilibria. Since $b$ is the global maximum of $f$, any mild equilibrium must contain $b$. Therefore, $\{b\}$ is the smallest mild equilibrium, i.e., $S^*=\{b\}$. It then follows from Lemma \ref{lm.optimal.mild} that $\{b\}$ is optimal mild. Moreover, by \eqref{eq.eg1.f} and \eqref{eq.cd} again, we have that $f(x)<J(x, \{b\})$ for any $x\neq b$, which implies that $\{b\}$ is the unique optimal mild equilibrium.
	
Now we verify that both $\{b\}$ and $\{a,b\}$ are strong equilibria. As for the optimal mild equilibrium $\{b\}$, a direct calculation from \eqref{e35} shows that
$$J'(b-,\{b\})=d\int_0^\infty e^{-s}\sqrt{2\beta s}ds>0,$$
	and by symmetry, we have $J'(b+,\{b\})<0$. Then,
	$$V_x(0,b-,\{b\})-V_x(0,b+,\{b\})=J'(b-,\{b\})-J'(b+,\{b\})>0.$$
	Meanwhile, by Lemma \ref{lm.v.c12}(a), we have $\Lc V(t,x\pm,\{b\})\equiv 0$ on $\X$. Therefore, we have $\Sk_{\{b\}}=\{b\}$ from \eqref{eq.define.Sk}. Since $\{b\}$ is closed and admissible, Theorem \ref{thm.strong.optimal}(b) tells that $\{b\}$ is a strong equilibrium. 
	Now consider the mild equilibirum $\{a,b\}$. Direct calculations from \eqref{e35'} show that
	$$J'(a-,\{a,b\})=c\int_0^\infty e^{-s}\sqrt{2\beta s}ds=\sqrt{\frac{\pi \beta }{2}} c,$$
and for any $x\in (a,b)$,
	\be\label{eq.counter.ab} 
	J'(x,\{a,b\})=-c\int_0^\infty e^{-s}\frac{\cosh((b-x)\sqrt{2\beta s})\sqrt{2\beta s}}{\sinh((b-a)\sqrt{2\beta s})}ds+d\int_0^\infty e^{-s}\frac{\cosh((x-a)\sqrt{2\beta s})\sqrt{2\beta s}}{\sinh((b-a)\sqrt{2\beta s})}ds.
	\ee
	By taking $x=a+$ in \eqref{eq.counter.ab} and combining with the first inequality in \eqref{eq.counter.abcd}, we have that
	\begin{align*}
		J'(a+,\{a,b\})=&-c\int_0^\infty e^{-s}\coth((b-a)\sqrt{2\beta s})\sqrt{2\beta s}ds+d\int_0^\infty e^{-s}\frac{\sqrt{2\beta s}}{\sinh((b-a)\sqrt{2\beta s})}ds\\
		<&\sqrt{\frac{\pi \beta }{2}} c=J'(a-,\{a,b\}).
	\end{align*}
	By taking $x=b-$ in \eqref{eq.counter.ab} and the fact that $0<c<d$, we have that
	\bee 
	\begin{aligned}
		J'(b-,\{a,b\})=&-c\int_0^\infty e^{-s}\frac{\sqrt{2\beta s}}{\sinh((b-a)\sqrt{2\beta s})}ds+d\int_0^\infty e^{-s}\coth((b-a)\sqrt{2\beta s})\sqrt{2\beta s}ds\\
		>&0>J'(b+,\{b\})=J'(b+,\{a,b\}).
	\end{aligned}
	\eee
	Hence, 
	\be\label{eq.eg1.ab} 
	V_x(0,x-,\{a,b\})>V_x(0,x+,\{a,b\}), \text{ for both $x=a,b$.}
	\ee 
	Meanwhile, Lemma \ref{lm.v.c12}(a) tells that $\Lc V(t,x,\{a,b\})\equiv 0$ on $\X\setminus \{a,b\}$. Therefore, by Theorem \ref{prop.weak.cha}, $\{a,b\}$ is a weak equilibrium. Moreover, by \eqref{eq.eg1.ab} and \eqref{eq.define.Sk}, $\Sk_{\{a,b\}}=\{a,b\}$. It then follows from Theorem \ref{thm.weak.strong} that $\{a,b\}$ is a strong equilibrium.
\end{proof}

\subsection{An example showing strong $\subsetneqq$ weak}
In this subsection we give an example in which a weak equilibrium is not strong, and thus $\{$strong equilibria$\}\subsetneqq \{$weak equilibria$\}$. Let $X$ be a geometric Brownian motion:
\begin{equation}\label{e23}
dX_t=\mu X_tdt+\sigma X_tdW_t
\end{equation}
with $\mathbb{X}=(0,\infty)$. Let $\delta(t)=\frac{1}{1+\beta t}$ and $f(x)=x\wedge K$ for some constant $K>0$. Assume that $\mu=\beta>0$.
\begin{Proposition}
$(0,\infty)$ is a weak equilibrium but not strong, while $[K,\infty)$ is the unique optimal mild equilibrium and a strong equilibrium.
\end{Proposition}
\begin{proof}
We first verify the result for $(0,\infty)$. Notice that $V(t,x,(0,\infty))\equiv \delta(t)f(x)$, then direct calculations show
\begin{align*}
\begin{cases}
	\mathcal{L} V(0,x,(0,\infty))=(-\beta +\mu)x =  0, \quad 0<x<K,\\
	\mathcal{L} V(0,x,(0,\infty))=-\beta K<0, \qquad \quad \;x>K,\\
	V_x(0,K-,(0,\infty))=1>0=V_x(0,K+,(0,\infty)).&
\end{cases}
\end{align*}
Therefore, by Theorem \ref{prop.weak.cha}, $(0,\infty)$ is a weak equilibrium. For $x\in(0,K)$, we have that for $\eps>0$ small enough, 
\begin{align*}
	\E[\delta(\rho^\eps_{(0,\infty)})f(X_{\rho_{(0,\infty)}^\eps})]&>\E\left[\delta(\eps)X_\eps 1_{\{X_\eps\leq K\}}\right]=\frac{e^{\mu\eps}}{1+\mu\eps}N(d_\eps)\cdot x\\
	&\geq\left(1+\mu\eps+\frac{1}{2}\mu^2\eps^2+o(\eps^2)\right)\left(1-\mu\eps+\mu^2\eps^2+o(\eps^2)\right)\left(1-\frac{1}{\sqrt{2\pi}d_\eps}e^{-d_\eps^2/2}\right)\cdot x\\
	&\geq\left(1+\frac{1}{2}\mu^2\eps^2+o(\eps^2)\right)(1+o(\eps^2))\cdot x=\left(1+\frac{1}{2}\mu^2\eps^2+o(\eps^2)\right)x>x,\\
\end{align*}
where
$d_\eps:=\frac{\ln(K/x)-(\mu+\frac{1}{2}\sigma^2)\eps}{\sigma\sqrt{\eps}}.$
This indicates that $(0,\infty)$ is not a strong equilibrium.

Now we verify the result for $[K,\infty)$. By Ito's formula,
$$
d\left(\frac{X_t}{1+\beta t}\right)=\frac{X_t}{1+\beta t}\left(-\frac{\beta}{1+\beta t}+\mu\right)dt+\frac{\sigma X_t}{1+\beta t} dW_t=\frac{\mu^2t X_t}{(1+\mu t)^2}dt+\frac{\sigma X_t}{1+\mu t} dW_t.
$$
Then, by the facts that $\rho_{[K,\infty)}>0$ $\P^x$-a.s. and $X_t>0$ for $x\in (0,K)$, we have that
\be\label{eq.eg2.1} 
J(x, [K,\infty))-f(x)=\E^x\left[\frac{X_{\rho_{[K,\infty)}}}{1+\beta\rho_{[K,\infty)})}\right]-x=\E^x\left[\int_0^{\rho_{[K,\infty)}}\frac{\mu^2t X_t}{(1+\mu t)^2}dt\right]>0,\quad \forall x\in(0,K), 
\ee
which shows that $[K,\infty)$ is a mild equilibrium. On the other hand, since $[K,\infty)$ is the set of global maxima of $f$, any mild equilibrium $S$ must contain $[K,\infty)$, for otherwise, $J(x,S)<K=f(x)$ for any $x\in S^c\cap [K,\infty)$, a contradiction. Therefore, $[K,\infty)$ is the smallest mild equilibrium and thus optimal. Now for any mild equilibrium $S$ such that $S\setminus [K,\infty)\neq \emptyset$,  \eqref{eq.eg2.1} indicates that $f(x)<J(x, [K,\infty))$ on $S\setminus [K,\infty)$, which implies that $S$ is not an optimal mild equilibrium. Hence, $[K,\infty)$ is the unique optimal mild equilibrium. Moreover, direct calculation shows that 
$$\Lc V(0,x+,[K,\infty))=\Lc (\delta(t)K)=-\beta K,\quad \forall x\in [K,\infty),$$
which tells that $\Sk_{[K,\infty)}=[K,\infty)$. Then by Theorem \ref{thm.strong.optimal}(b), $[K,\infty)$ is also strong.
\end{proof}

\subsection{Stopping of an American put option}\label{subsec:eg.GBM}
Consider the American put example in \cite[Section 6.3]{MR4116459}. In particular, $X$ is a geometric Brownian motion given by \eqref{e23} with $\X:=(0,\infty)$. Let $\mu\geq0$. The payoff function is defined as $f(x):=(K-x)^+$, and $\delta(t):=\frac{1}{1+\beta t}$. We shall provide all three types of equilibria. To begin with, the following lemma summarizes the results of mild equilibria stated in Lemma 6.11, Corollary 6.13 and Proposition 6.15 in \cite{MR4116459}.

\begin{Lemma}\label{lm.eg2}
	\bi 
	\item[(i)] If $S$ is a mild equilibrium, then $S\cap (0,K]=(0,a]$ for some $a\in(0,K]$.
	\item[(ii)] $S=(0,a]\subset (0,K]$ is mild equilibrium if and only if $a\geq \frac{\lambda}{1+\lambda}K$, where 
	\be\label{eq.eg3.nulambda}  
	\lambda:=\int_0^\infty e^{-s}\Big( \sqrt{\nu^2+2\beta s/\sigma^2} +\nu \Big)>0,\quad  \nu:=\frac{u}{\sigma^2}-\frac{1}{2}.
	\ee
	\item[(iii)] $S^*=(0, \frac{\lambda }{1+\lambda}K]$ is the intersection of all mild equilibria and is the unique optimal mild equilibria.
	\ei 
\end{Lemma}

Following from Lemma \ref{lm.eg2}(ii), we shall call the mild equilibria that belong to the family $\{ (0,a]: a\geq \frac{\lambda}{1+\lambda}K\}$ are ``type I" mild equilibria.\footnote{It contains the trivial mild equilibrium $\X$ by setting $a=\infty$ and $\X=\X\cap(0,\infty]$.} The following proposition shows that, except ``type I" mild equilibria, all other mild equilibria take the same form: $(0,a]\cup D$ that satisfies a certain condition, and we shall call this family of mild equilibria ``type II" mild equilibria.

\begin{Proposition}\label{prop.eg2}
	Except the ``type I" mild equilibria in Lemma \ref{lm.eg2}(ii), all other mild equilibria take form: $(0,a]\cup D$ such that
	\be\label{eq.eg2.mild} 
	-(K-a)\int_0^\infty e^{-s}
	\left( \dfrac{\nu}{a}+\dfrac{\sqrt{\nu^2+2\beta s/\sigma^2}}{a}\cdot \dfrac{(b/a)^{\sqrt{\nu^2+2\beta s/\sigma^2}}+(a/b)^{\sqrt{\nu^2+2\beta s/\sigma^2}}}{(b/a)^{\sqrt{\nu^2+2\beta s/\sigma^2}}-(a/b)^{\sqrt{\nu^2+2\beta s/\sigma^2}}}\right)ds\geq -1,
	\ee
	where $D$ is a closed subset of $[K,\infty)$ and $b:=\inf\{x\in D\}$ satisfying $b>a$.
\end{Proposition}

\begin{proof}
	Lemma \ref{lm.eg2}(i)(ii) together imply that any mild equilibrium is either of type I or takes the form: $(0,a]\cup D$ with $D$ being a closed subset of $[K,\infty)$. Consider a closed set of such form $S=(0,a]\cup D$ with $b:=\inf\{x\in D\}>a$. When $a\geq K$, the fact that $f=0$ on $[K,\infty)$ immediately gives that $S$ is a mild equilibrium. Notice that $-(K-a)\geq 0$ and the integrand in the LHS of \eqref{eq.eg2.mild} is positive, so \eqref{eq.eg2.mild} holds.
	
	When $a<K$, we have
	\begin{align*}
		J(x,S)&=(K-a)\int_0^\infty \frac{p(t)}{1+\beta t} dt=(K-a)\int_0^\infty e^{-s}\E^x[e^{-\beta s \tau_{(a,b)}}\cdot 1_{\{X_{\tau_{(a,b)}}=a\}}]ds\\
		&=  (K-a)\int_0^\infty e^{-s}\left(\frac{a}{x}\right)^\nu \dfrac{(b/x)^{\sqrt{\nu^2+2\beta s/\sigma^2}}-(x/b)^{\sqrt{\nu^2+2\beta s/\sigma^2}}}{(b/a)^{\sqrt{\nu^2+2\beta s/\sigma^2}}-(a/b)^{\sqrt{\nu^2+2\beta s/\sigma^2}}}ds,\quad \forall x\in (a,b),
	\end{align*}
	where $p(t):=\P^x(\tau_{(a,b)}\in dt,\, X_{\tau_{(a,b)}}=a)$, and the second line above follows from \cite[3.0.5 (a) on page 633]{MR1912205}. Direct calculations show that for any $x\in (a,b)$
	\be\label{eq.eg2.deriv}
		\begin{aligned}
			J'(x, S)=& -(K-a)\int_0^\infty e^{-s}\Bigg( \frac{a^\nu \nu}{x^{\nu+1}}\cdot \dfrac{(b/x)^{\sqrt{\nu^2+2\beta s/\sigma^2}}-(x/b)^{\sqrt{\nu^2+2\beta s/\sigma^2}}}{(b/a)^{\sqrt{\nu^2+2\beta s/\sigma^2}}-(a/b)^{\sqrt{\nu^2+2\beta s/\sigma^2}}}\\
			&+\frac{a^\nu \sqrt{\nu^2+2\beta s/\sigma^2}}{x^{\nu+1}}\cdot \dfrac{(b/x)^{\sqrt{\nu^2+2\beta s/\sigma^2}}+(x/b)^{\sqrt{\nu^2+2\beta s/\sigma^2}}}{(b/a)^{\sqrt{\nu^2+2\beta s/\sigma^2}}-(a/b)^{\sqrt{\nu^2+2\beta s/\sigma^2}}}\Bigg)ds,
		\end{aligned}
\ee
	\be\label{eq.eg2.deriv'}
\begin{aligned}
			J''(x, S)=&(K-a)\int_0^\infty e^{-s}\Bigg(\frac{a^\nu (2\nu^2+\nu+2\beta s/\sigma^2)}{x^{\nu+2}} \cdot \dfrac{(b/x)^{\sqrt{\nu^2+2\beta s/\sigma^2}}-(x/b)^{\sqrt{\nu^2+2\beta s/\sigma^2}}}{(b/a)^{\sqrt{\nu^2+2\beta s/\sigma^2}}-(a/b)^{\sqrt{\nu^2+2\beta s/\sigma^2}}}\\
			&+ \frac{a^\nu (2\nu+1)\sqrt{\nu^2+2\beta s/\sigma^2}}{x^{\nu+2}} \cdot \dfrac{(b/x)^{\sqrt{\nu^2+2\beta s/\sigma^2}}+(x/b)^{\sqrt{\nu^2+2\beta s/\sigma^2}}}{(b/a)^{\sqrt{\nu^2+2\beta s/\sigma^2}}-(a/b)^{\sqrt{\nu^2+2\beta s/\sigma^2}}} \Bigg)ds.
		\end{aligned}
	\ee
	Recall $\nu$ in \eqref{eq.eg3.nulambda}, we have that
	$$
	\nu + \sqrt{\nu^2+2\beta s/\sigma^2}>0 \quad \text{and}\quad (2\nu^2+\nu+2\beta s/\sigma^2)+\left((2\nu+1)\sqrt{\nu^2+2\beta s/\sigma^2}\right)>0.
	$$
	This together with 
	$$0<(b/x)^{\sqrt{\nu^2+2\beta s/\sigma^2}}-(x/b)^{\sqrt{\nu^2+2\beta s/\sigma^2}}<(b/x)^{\sqrt{\nu^2+2\beta s/\sigma^2}}+(x/b)^{\sqrt{\nu^2+2\beta s/\sigma^2}}$$
	implies that
	both the integrands on the RHS of \eqref{eq.eg2.deriv} and \eqref{eq.eg2.deriv'} are positive. Therefore, $J'(x,S)<0$ and $J''(x,S)>0$ for $x\in (a,b)$, and thus $J(x,S)$ is strictly decreasing and convex on $(a,b)$. This together with the shape of $f$ on $(a,b)$ indicates that $S$ is a mild equilibrium if and only if $J'(a+,S)\geq -1$. From \eqref{eq.eg2.deriv}, we have
	\bee 
	J'(a+, S)=-(K-a)\int_0^\infty e^{-s}
	\left( \dfrac{\nu}{a}+\dfrac{\sqrt{\nu^2+2\beta s/\sigma^2}}{a}\cdot \dfrac{(b/a)^{\sqrt{\nu^2+2\beta s/\sigma^2}}+(a/b)^{\sqrt{\nu^2+2\beta s/\sigma^2}}}{(b/a)^{\sqrt{\nu^2+2\beta s/\sigma^2}}-(a/b)^{\sqrt{\nu^2+2\beta s/\sigma^2}}}\right)ds,
	\eee
	so $S$ is a mild equilibrium if and only if \eqref{eq.eg2.mild} holds.
	Notice that $J'(a+, S)$ converges to $0$ when $a\nearrow K$. Then for any $b>K$, by the continuity of function $a\mapsto J'(a+, S)$, there exists a constant $a_b<K$ such that for all $a\in[a_b, K)$, \eqref{eq.eg2.mild} indeed holds and $S$ is a mild equilibrium. 
\end{proof}

\begin{Proposition}\label{prop.eg2.weakstrong}
	$S^*=(0,\frac{\lambda}{1+\lambda}K]$ is the unique weak and the unique strong equilibrium.
\end{Proposition}

\begin{proof}
	We first find all weak equilibria. Since a weak equilibrium is also mild, by Proposition \ref{prop.eg2}, it is sufficient to select weak equilibria from the two types of mild equilibria. Given a mild equilibrium $S$ that is weak, no matter which type it is, $S$ must not contain $K$. Otherwise, by Lemma \ref{lm.eg2}(i), $(0,K]\subset S$, which together with $f=0$ on $[K,\infty)$ implies that
	$$
	V_x(0,K-,S)=-1<0=V_x(0,K+,S),
	$$
	which contradicts \eqref{eq.weak.equicha1} in Theorem \ref{prop.weak.cha}. 
	
	 Consider an arbitrary type I mild equilibrium $(0,a]$ with $\frac{\lambda}{1+\lambda}K\leq a<K$. By the smooth-fit condition in Corollary \ref{cor.boundary.smoothfit}, $S=(0,a]$ is a weak equilibrium if and only if 
	$$V_x(0,a+,(0,a])=J'(a+, (0,a])=-1.$$
	From the calculation in the proof of Lemma 6.12 in \cite{MR4116459}, such condition is satisfied if and only if $a=\frac{\lambda }{1+\lambda}K$. Hence, $S^*=(0,\frac{\lambda }{1+\lambda}K]$ is the only weak equilibrium among the type I mild equilibria. Now pick any type II mild equilibrium $S=(0,a]\cup D$ with $b:=\inf\{x\in D\}$. As $K\notin S$, we have $a<K$. Then by \eqref{eq.eg2.deriv}, we have
	$$
	V_x(0,b-,S)=-\frac{K-a}{b^{\nu+1} }\int_0^\infty e^{-s}
	\dfrac{2a^\nu \sqrt{\nu^2+2\beta s/\sigma^2}}{(b/a)^{\sqrt{\nu^2+2\beta s/\sigma^2}}-(a/b)^{\sqrt{\nu^2+2\beta s/\sigma^2}}}ds<0=V_x(0,b+,S).
	$$
	That is, the smooth-fit condition fails at the boundary $x=b$, and hence $S$ is not weak. In sum, $S^*=(0,\frac{\lambda }{1+\lambda}K]$ is the unique weak equilibrium.
	
	Finally, a direct calculation shows that 
	$$\Lc V(0, x-, S^*)=-\beta (K-x)-\mu x<0,\quad \forall x\in \left(0, \frac{\lambda }{1+\lambda}K\right],$$
	so $S^*=\Sk_{S^*}$. Then, by Theorem \ref{thm.weak.strong} and the fact that $S^*$ is the unique weak equilibrium, we can conclude that $S^*$ is the unique strong equilibrium.
\end{proof}

\begin{Remark}\label{rm:eg.admissible}
Within this example, we do not restrict equilibria to be admissible. The unique weak, strong, optimal mild equilibrium $(0,\frac{\lambda}{1+\lambda}K]$ turns out to be indeed admissible. Moreover, type I mild equilibria are all admissible, while any type II mild equilibrium $S=(0,a]\cup D$ with $b:= \inf\{x\in D\}>a$ has an alternative $(0,a]\cup [b,\infty)$, which share the same $J$ value and is admissible.
\end{Remark}

\appendix
\section{Proof for Results in Section \ref{sec:notations}}\label{sec:appendix}

\begin{proof}[{\bf Proof of \eqref{e2} in Remark \ref{rm.close.regular}}]
	Let $X_0=x\in\X$ and $h>0$ be small enough such that $[x-h,x+h]\subset\X$. Let $Y=(Y_t)_{t\geq 0}$ follows $dY_t=\hat\mu(Y_t)dt+\hat\sigma(Y_t)dW_t$ with $Y_0=x$,
	where
	$$\hat\mu(y):=\begin{cases}
		\mu(y),& x-h\leq y\leq x+h,\\
		\mu(x-h),& y<x-h,\\
		\mu(x+h),&y>x+h,\end{cases}\quad\text{and}\quad \hat\sigma(y):=\begin{cases}
		\sigma(y),& x-h\leq y\leq x+h,\\
		\sigma(x-h),& y<x-h,\\
		\sigma(x+h),&y>x+h.\end{cases}
	$$
	Then by \cite[Lemma A.1]{MR4067078}, for any $t>0$,
	$$\P^x\left(\max_{0\leq s\leq t}Y_s>x\right)=\P^x\left(\min_{0\leq s\leq t}Y_s<x\right)=1.$$
	Note that $Y_s=X_s$ for $s\leq\tau_{B(x,h)}$. Then for a.s. $\omega\in\{\tau_{B(x,h)}>1/n\}$,
	$$\max_{0\leq s\leq t}X_s(\omega)>x\quad\text{and}\quad\min_{0\leq s\leq t}X_s(\omega)<x,\quad\forall\,t\in(0,1/n)\ \text{and thus}\ \forall\, t>0.$$
	Then \eqref{e2} follows from the arbitrariness of $n\in\mathbb{N}$.	

\end{proof}

\begin{proof}[{\bf Proof of Lemma \ref{lm.delta.inequ}}]
By \eqref{eq.assume.DI}, for any $t,r\geq 0$
$$
\delta(t+r)-\delta(t)\geq \delta(t)(\delta(r)-\delta(0)),
$$
This together with the differentiability of $\delta(t)$ implies that
$
\delta'(t)\geq \delta(t)\delta'(0).
$
As $\delta'(t)\leq 0$, 
$$
1-\delta(t)=\int_0^t -\delta'(s)ds\leq \int_0^t -\delta(s)\delta'(0)ds\leq \int_0^t |\delta'(0)|ds=|\delta'(0)|t.
$$
\end{proof}

\begin{proof}[{\bf Proof of Lemma \ref{rm.suff.newassume}}] Take an admissible stopping policy $S$. Let $a,b\in \X$ such that $[a,b]\subset\X$ and $(a,b)\subset S^c$. Throughout the proof, $C>0$ will serve as a generic constant that may change from one line to another and is independent of $r$.

Set $v(x,r,S):=\E^x[e^{-r \rho_S}f(X_{\rho_S})]$. We first provide an estimate for $|v_x(x,r,S)|+|v_{xx}(x,r,S)|$ on $[a,b]$. Assumption \ref{assume.x.elliptic}(i) and the boundedness of $f$ gives the well-posedness of $v(\cdot, r, S)$ for all $r\geq 0$, 
	and
	\be\label{eq.bound.vr}
	\sup_{x\in[a,b], r\geq 0} |v(x,r,S)|\leq \sup_{x\in[a,b]} |v(x,0,S)|\leq C.
	\ee
	For an arbitrary $r\geq 0$, by a standard probabilistic argument,
	one can derive that $v(x,r,S)\in \Cc^2([a,b])$ satisfies the following elliptic equation 
	\be\label{eq.r.elliptic} 
	\begin{cases}
		-ru(x)+\mu(x)u'(x)+\frac{1}{2}\sigma^2(x)u''(x)=0, \quad x\in (a,b),\\
		u(a)=v(a,r,S), \; u(b)=v(b,r,S).
	\end{cases}
	\ee 
	Recall the strictly increasing function $y=\phi(x)$ defined in \eqref{eq.lm.change} and denote by $\phi^{-1}$ the inverse function of $\phi$. Define function $\tilde{u}(y):=u(\phi^{-1}(y))$ (i.e., $u(x)=\tilde{u}(\phi(x))$) on $[\phi(a), \phi(b)]$. Then $\tilde{u}\in \Cc^2([\phi(a), \phi(b)])$, and \eqref{eq.r.elliptic} leads to 
	\be\label{eq.tildeu.elliptic} 
	\begin{cases}
		-r\tilde{u}(y)+\frac{1}{2}\tilde{\sigma}^2(y)\tilde{u}''(y)=0, \quad y\in (\phi(a),\phi(b)),\\
		\tilde{u}(\phi(a))=v(a,r,S), \; \tilde{u}(\phi(b))=v(b,r,S),
	\end{cases}
	\ee
	with $\tilde{\sigma}(y):=\sigma(\phi^{-1}(y))\phi'(\phi^{-1}(y))$. 
	Then \eqref{eq.bound.vr} together with the maximum principle implies that
	$$
	\sup_{y\in (\phi(a),\phi(b))} |\tilde{u}(y)|\leq v(a,r,S)\vee v(b,r,S)\leq C \qquad \forall r\geq 0.
$$
	This together with the fact that $\tilde{u}''=2\tilde{u}r /(\tilde{\sigma}^2)$ and uniform ellipticity of $\tilde{\sigma}$ on $(\phi(a),\phi(b))$ leads to
\be\label{eq.u''.bound}
	\sup_{y\in (\phi(a),\phi(b))} |\tilde{u}''(y)|\leq r C \qquad \forall r\geq 0.
\ee	
By the mean value theorem, there exists $y_0\in (\phi(a),\phi(b))$ such that 
\be\label{eq.mvt} 
|\tilde{u}'(y_0)|=\left|\frac{v(b,r,S)-v(a,r,S)}{\phi(b)-\phi(a)}\right|\leq \frac{2C}{\phi(b)-\phi(a)} \qquad \forall r\geq 0.
\ee
Then by \eqref{eq.u''.bound} and \eqref{eq.mvt}, for any $y\in (\phi(a), \phi(b))$ and $r\geq 0$, we have
$$
|\tilde{u}'(y)|\leq |\tilde{u}'(y_0)|+\int_{y_0}^y |\tilde{u}''(l)|dl\leq \frac{2C}{\phi(b)-\phi(a)}+\int_{\phi(a)}^{\phi(b)}rCdl=\frac{2C}{\phi(b)-\phi(a)}+rC(\phi(b)-\phi(a)).
$$
Therefore,
	$$
	\sup_{y\in (\phi(a),\phi(b))} (|\tilde{u}'(y)|+|\tilde{u}''(y)|)\leq C(1+r)\qquad \forall r\geq 0.
	$$
	This together with the fact that 
	$$\sup_{x\in(a,b)} (|u'(x)|+|u''(x)|)\leq\sup_{y\in (\phi(a),\phi(b))} (|\tilde{u}'(y)|+|\tilde{u}''(y)|)\cdot \sup_{x\in (a,b)}(|\phi'(x)|+|\phi'(x)|^2+|\phi''(x)|)$$
	 implies
	\be\label{eq.vrx.finite}
	\sup_{x\in(a,b)} (|v_x(x,r,S)|+|v_{xx}(x,r,S)|)=\sup_{x\in(a,b)} (|u'(x)|+|u''(x)|)\leq C(1+r) \qquad \forall r\geq 0.
	\ee	 
	
	Next, we verify that $V\in \Cc^{1,2}([0,\infty)\times [a,b])$. For any $r\geq 0$, $v_x(a+,r,S)$, $v_x(b-,r,S)$ (resp. $v_{xx}(a+,r,S)$, $v_{xx}(b-,r,S)$) all exist and satisfy the same bound as the RHS of \eqref{eq.vrx.finite}. Hence, we conclude that $v(x,r,S)\in \Cc^2([a,b])$. By Fubini theorem, \eqref{eq.weight.delta} leads to 
	\be\label{eq.lm.fubini} 
	V(t,x,S)=\int_0^\infty e^{-rt}v(x,r,S)dF(r)\quad \forall x\in\X.
	\ee 
	This, together with \eqref{eq.vrx.finite} and the assumption $\int_0^\infty rdF(r)<\infty$, implies that $V\in \Cc^{1,2}([0,\infty)\times [a,b])$. 
	
	Finally, we prove \eqref{eq.vx.lineart} for $V_x(t,x+,S)$ (the verification for $V_x(t,x-,S)$ is similar and thus omitted). For $(t,x)\in [0,\infty)\times [a,b)$, by \eqref{eq.vrx.finite}, \eqref{eq.lm.fubini} and the assumption $\int_0^\infty rdF(r)<\infty$,
\begin{align}\label{eq.vx.s} 
V_x(t,x+,S)=\int_0^\infty e^{-rt}v_x(x+,r,S)dF(r). 
\end{align}	
Now take any $x\in\X$. If there exists some $h>0$ such that $(x,x+h)\subset S^c$, then by \eqref{eq.weight.at}, \eqref{eq.vrx.finite} and \eqref{eq.vx.s}, we have that
	\begin{align*}
		|V_x(t,x+,S)-V_x(0,x+,S)|&\leq  \int_0^\infty|v_x(x,r,S)| |e^{-rt}-1|dF(r)\\
		&\leq C \int_0^\infty (1+r)(1-e^{-rt})dF(r)=o(\sqrt{t}),\ \text{as}\ t\to0.
	\end{align*}
	Otherwise, since $S$ is admissible, there exists some $\bar{h}>0$ such that $(x,x+\bar{h})\subset S$. Then 
	$$|V_x(t,x+,S)-V_x(0,x+,S)|=|\delta(t)f(x)-\delta(0)f(x)|\leq|\delta'(0)|tf(x)=o(\sqrt{t}),\ \text{as}\ t\to0,$$
	where the above inequality follows from Lemma \ref{lm.delta.inequ}. In sum, \eqref{eq.vx.lineart} holds for $V_x(t,x+,S)$.
	\end{proof}

\begin{proof}[{\bf Proof of Lemma \ref{lm.v.c12}}]
	{\bf Part (a):} Assumption \ref{assume.new} guarantees that $V(t,x,S)\in \Cc^{1,2}([0,\infty)\times \overline{S^c})$. \eqref{e2} implies that $\P^{x}(\rho_S=0)=1$ for any $x\in S$, so $V(t,x,S)=\delta(t)f(x)$ for any $(t,x)\in [0,\infty)\times S$. Now we prove that
	\be\label{eq.lm.parta} 
	\Lc V(t,x,S)\equiv 0\quad \forall (t,x)\in [0,\infty)\times S^c.
	\ee 
	Take $(t,x_0)\in [0,\infty)\times S^c$. Since $S^c$ is open, we can take $h>0$ such that $[x_0-h,x_0+h]\subset S^c$. By Assumption \ref{assume.x.elliptic}(i) and $V(t,x,S)\in \Cc^{1,2}([0,\infty)\times \overline{S^c})$, $(s,x)\mapsto \Lc V(s,x,S)$ is continuous on the compact set $[t,t+1]\times\overline{B(x_0,h)}$. Then 
	\be\label{eq.bound}
	\sup_{(s,x)\in [t,t+1]\times \overline{B(x_0,h)}} |\Lc V(s,x,S)|<\infty.
	\ee
Applying Ito's formula to $V(t+s,X_s,S)$ and taking expectation, the diffusion term vanishes due to the boundedness of $V_x \sigma$ on $[t,t+1]\times \overline{B(x_0,h)}$, we have that
	\be\label{eq.generator.lv}   
	\begin{aligned}
		\E^{x_0}[V(t+\varepsilon\wedge \tau_{B(x_0,h)},X_{ \varepsilon\wedge \tau_{B(x_0,h)}},S)]-V(t,x_0,S)	=\E^{x_0}\left[\int_0^{\varepsilon\wedge\tau_{B(x_0,h)}}\Lc V(t+s, X_s,S)ds\right].
	\end{aligned}
	\ee
Meanwhile, by the continuity of $\Lc V(s,x, S)$ on $[t,t+1]\times\overline{B(x_0,h)}$,
	$$
	\lim_{\varepsilon\searrow 0}\frac{1}{\varepsilon}\int_0^{\varepsilon\wedge \tau_{B(x_0,h)}}\Lc V(t+s, X_s,S)ds=\Lc V(t,x_0,S),\quad \P^{x_0}\text{-a.s.}.$$
	Thanks to \eqref{eq.bound}, we can apply the dominated convergence theorem to above equality and get that
\begin{equation}\label{eq.generator.lvlm}    
	\lim_{\varepsilon\searrow 0}\frac{1}{\varepsilon}\E^{x_0}\Big[\int_0^{\varepsilon\wedge \tau_{B(x_0,h)}}\Lc V(t+s, X_s,S)ds\Big]=\Lc V(t,x_0,S).
\end{equation}
	On the other hand, for any $\varepsilon>0$, it is obvious that
	\be\label{eq.vtau.varep}
	\begin{aligned}
		\E^{x_0}[V(t+\varepsilon\wedge \tau_{B(x_0,h)},X_{ \varepsilon\wedge \tau_{B(x_0,h)}},S)]= \E^{x_0}[\delta(t+\rho_S)f(X_{\rho_S})]=V(t,x_0,S).
	\end{aligned}
	\ee
	Then by \eqref{eq.generator.lv}--\eqref{eq.vtau.varep}, we have that
	\bee  
	0=\lim_{\varepsilon\searrow 0}\frac{1}{\varepsilon}\left( \E^{x_0}[V(t+\varepsilon\wedge \tau_{B(x_0,h)},X_{ \varepsilon\wedge \tau_{B(x_0,h)}},S)]-V(t,x_0,S)\right)
	=\Lc V(t,x_0,S),
	\eee
	and thus \eqref{eq.lm.parta} holds.
	
	{\bf Part (b):} The existence of $\Lc V(t,x\pm,S)$ on $[0,\infty)\times \X$ follows from part (a), the differentiability of $\delta$ and Assumption \ref{assume.f.positive}(ii). Take $x_0\in \X$ and  $h>0$ such that $[x_0-h,x_0+h]\subset \X$. We show that
	\be\label{eq.lm.b} 
	\sup\limits_{(t,x)\in [0, \infty)\times \overline{B(x_0, h)}}|\Lc V(t,x-,S)|<\infty,
	\ee 
	and the result for $\Lc V(t,x+,S)$ follows from a similar argument.
	Let $x\in B(x_0, h)$. If $(x-h',x)\in S^c$ for some constant $h'>0$, 
	then by the left continuity of $y\mapsto \Lc V(t,y-,S)$ at $y=x$ and \eqref{eq.lm.lv0} in part (a), we have $\Lc V(t,x-,S)=0$. Otherwise, since $S$ is admissible, there exists $\bar{h}\in(0,h)$ such that $(x-\bar{h}, x)\subset S$, then part (a) tells that $V(t,x,S)=\delta(t)f(x)$ on $[0,\infty)\times (x-\bar{h}, x)$, and we have that
	\begin{align*}
		|\Lc V(t,x-,S)|= &\left|\delta'(t) f(x)+\delta(t)\left(b(x)f'(x-)+\frac{1}{2}\sigma^2(x) f''(x-)\right)\right|\\
		\leq &\sup\limits_{y\in \overline{B(x_0, h)}}
		\left(|\delta'(0)||f(y)|+|b(y)f'(y-)|+\frac{1}{2}\sigma^2(y) |f''(y-)| \right)
		<\infty,
	\end{align*}
where the inequality above follows from the first inequality in Lemma \ref{lm.delta.inequ}, Assumptions \ref{assume.x.elliptic}(i) and \ref{assume.f.positive}(ii). Hence, \eqref{eq.lm.b} holds.

\end{proof}

\bibliographystyle{plain}
\bibliography{reference}

\end{document}